\documentclass[onefignum,onetabnum]{siamart190516}

\usepackage{lipsum}
\usepackage{amsfonts}
\usepackage{amssymb}
\usepackage{graphicx}
\usepackage{epstopdf}
\usepackage{algorithmic}
\usepackage{multirow}
\usepackage{caption}
\usepackage{subcaption}

\usepackage{geometry}
\usepackage{bbm}

\DeclareMathOperator*{\argmax}{arg\,max}

\newenvironment{mat}{\left[\begin{array}{ccccccccccccccc}}{\end{array}\right]}
\newcommand\bcm{\begin{mat}}
\newcommand\ecm{\end{mat}}

\ifpdf
  \DeclareGraphicsExtensions{.eps,.pdf,.png,.jpg}
\else
  \DeclareGraphicsExtensions{.eps}
\fi

\newsiamremark{remark}{Remark}
\newsiamremark{hypothesis}{Hypothesis}
\crefname{hypothesis}{Hypothesis}{Hypotheses}
\newsiamthm{claim}{Claim}

\headers{Truncated Orthogonal Iteration for Sparse Eigenvector Problems}{H. Liu}

\title{Analysis of Truncated Orthogonal Iteration for Sparse Eigenvector Problems\thanks{Submitted to the editors DATE.
\funding{This work was funded by }}}

\author{Hexuan Liu\thanks{Department of Applied Mathematics, University of Washington, Seattle, WA
  (\email{hl65@uw.edu}, \email{saravkin@uw.edu}).}
\and Aleksandr Aravkin\footnotemark[2]}

\usepackage{amsopn}
\DeclareMathOperator{\diag}{diag}

\ifpdf
\hypersetup{
  pdftitle={Analysis of Truncated Orthogonal Iteration for Sparse Eigenvector Problems},
  pdfauthor={Hexuan Liu}
}
\fi

\begin{document}

\maketitle

\begin{abstract}
A wide range of problems in computational science and engineering 
require estimation of sparse eigenvectors for high dimensional systems. 
Here, we propose two variants of the Truncated Orthogonal Iteration to compute multiple leading eigenvectors with sparsity constraints simultaneously. We establish numerical convergence results for the proposed algorithms using a perturbation framework, and extend our analysis to other existing alternatives for sparse eigenvector estimation. We then apply our algorithms to solve the sparse principle component analysis problem for a wide range of test datasets, from simple simulations to real-world datasets including MNIST, sea surface temperature and 20 newsgroups. In all these cases, we show that the new methods get state of the art results quickly and with minimal parameter tuning. 

\end{abstract}



\section{Introduction}
Sparse eigenvector problems arise in many applications where localized and structured eigenvectors are desired, such as sparse principal component analysis (sparse PCA), sparse dictionary learning and densest $k-$subgraphs recovery. In sparse PCA, sparse loading vectors have better interpretability, since each principal component is a linear combination of only a few of the original features. The goal of sparse coding/dictionary learning is to represent the input signal as a sparse linear combination of the dictionary elements. The densest $k-$subgraph can also be formulated as a sparse eigenvector problem~\cite{yuan2013truncated} to find a set of $k$ vertices with maximum average degree in the subgraph induced by the set.
To formalize the problem, we assume that the leading eigenvectors corresponding to the largest $m$ eigenvalues of a positive semidefinite matrix $\bar{A}$ are sparse, i.e.
\begin{equation}\label{eq:sparse}
    \bar{P} = \argmax_{P^TP=I_m}\text{Tr}(P^T\bar{A}P), \quad \bar{P} \text{ is sparse.}
\end{equation}
Here, {\it sparse} means that each column of $P$ has a lot of entries that are either exactly zero or sufficiently close to zero. In practice, $\bar{A}$ is often unknown and we are given a perturbed positive semidefinite matrix $A$: $A=\bar{A}+E$. Our goal is to recover the true eigenvectors $\bar{P}$ from $A$. 

A straightforward formulation for the sparse eigenvector problem given $A$ is as follows:
\begin{equation}\label{eq:original}
    \bar{Q}=\argmax_{Q^TQ=I_m}\text{Tr}(Q^T{A}Q),  \quad \text{subject to }\|q_i\|_0\leq k_i , ~i=1, \ldots , m,
\end{equation}
where $q_i$ is the $i^{th}$ column vector of $Q$ and $\|\cdot\|_0$ denotes the $\ell_0$ ``norm" which is the number of nonzeros in a vector. To find a solution to~\cref{eq:original} is an NP-hard problem. Moreover, the solution for~\cref{eq:original} is not a good approximation for~\cref{eq:sparse} unless some additional assumptions are imposed on $E$, e.g. when $E=\sigma^2 I$. In this work we do not impose additional assumptions on $A$ and do not aim to solve~\cref{eq:original}. We focus instead on the orthogonal iteration for solving the standard eigenvalue problems, and present two algorithms. In the first algorithm, we relax the sparsity constraint, while in the second, we relax the orthogonality constraint. We analyze whether truncation at each iteration yields a better approximation to the true eigenvectors than those of the standard orthogonal iteration.

Many existing algorithms for sparse eigenvector recovery focus on recovering a single eigenvector and then using a deflation scheme to generalize to multiple components~\cite{yuan2013truncated, d2008optimal, shen2008sparse, jolliffe1989rotation, zou2006sparse}. The downside of this approach is that deflation adds extra perturbation error to the original problem, with estimates of latter components accumulating errors from each deflation. The deflation step itself can also be a computational bottleneck. When several leading eigenvalues are clustered, it is also difficult to identify the corresponding eigenvectors, and sometimes a subspace is preferred over individual eigenvectors. To avoid these issues, we use the orthogonal iteration, which is a block generalization of the power method and outputs an orthonormal basis of the subspace spanned by the leading eigenvectors.

One of the most widely used applications of sparse eigenvectors is sparse PCA, where $A$ is the empirical covariance matrix and $\bar{A}$ is the true covariance matrix. PCA~\cite{pearson1901liii, hotelling1933analysis, Jolliffe:1986} is one of the most widely used dimensionality reduction techniques. It is computed by doing an eigendecomposition on the sample covariance matrix, and finds a sequence of orthogonal vectors that estimate principal directions of the data variance. In high-dimensional settings, classical PCA suffers from inconsistency~\cite{johnstone2009consistency} and poor interpretability~\cite{d2005direct}, and over the last two decades many algorithms and theory for sparse PCA have been proposed to mitigate these issues, see e.g. the survey~\cite{zou2018selective}. In this paper we analyze the problem~\cref{eq:sparse} in the general sparse eigenvector setting and then focus on the sparse PCA application.

Another line of relevant work focuses on finding the row-sparse principal subspace, assuming multiple eigenvectors share the same sparsity pattern (or ``support set")~\cite{vu2013minimax, wang2020upper, tian2020learning}. This is equivalent to first selecting a sparse subset of features (an NP-hard problem~\cite{tian2020learning}) and then applying PCA. A practical limitation of the row-sparse formulation is that we cannot always assume the same support set across all leading eigenvectors of interest. For example, in face recognition tasks~\cite{jenatton2010structured}, brain imaging~\cite{de2017structured}, and natural language processing~\cite{zhang2012sparse}, the goal is often to find different localized and interpretable patterns in different eigenvectors. Here, we consider the general problem of column-sparse subspace estimation, without assuming a common support set.


{\bf Contributions.}
We propose a general framework for estimating sparse eigenvectors and differentiate between the deflation scheme and block scheme. We then develop two new algorithms based on the orthogonal iteration to obtain several leading eigenvectors with sparsity constraints. We provide a deterministic convergence analysis for methods within the block framework, without additional assumptions on the matrix $A$, and extend the analysis to other sparse eigenvector algorithms. Finally, we demonstrate the accuracy and efficiency of the proposed algorithms by applying them to simulated and real-world datasets, including the pitprops, sea surface temperature, MNIST, and 20 newsgroup datasets.  



\vspace{0.2cm}
\textbf{Notation:} Let $\mathbb{S}^p=\{A\in \mathbb{R}^{p\times p}|A=A^T\}$ denote the set of symmetric matrices. For any $A\in\mathbb{S}^p$, we denote its eigenvalues by $\lambda_{\min}(A)=\lambda_p(A)\leq \cdots \leq \lambda_1(A)=\lambda_{\max}(A)$.  We use $\rho(A)$ and $\|A\|_2$ to denote the spectral norm of $A$, which is $\max\{|\lambda_{\max}(A)|, |\lambda_{\min}(A)|\}$. For vectors, $\| \cdot \|_2$ or $\| \cdot \|$ will denote the 2-norm, while $\| \cdot \|_0$ denotes the $\ell_0$ ``norm" which is the number of nonzeros in a vector. Let $\lambda_{\max} (A, k)=\max_{x\in \mathbb{R}^p} x^TAx$, such that $\|x\|=1$, $\|x\|_0\leq k$. Define $\rho(E, k):=\sqrt{\lambda_{\max}(E^TE, k)}$. We use $\|A\|_F$ to denote the Frobenius norm of $A$.

\section{Methodology}
\label{sec:method}

Many methods for finding an approximate solution to \cref{eq:original} use a power iteration-like scheme, either using deflation to obtain the leading vectors sequentially~\cite{d2008optimal,yuan2013truncated}, or using block approaches to compute several vectors at once~\cite{journee2010generalized,ma2013sparse,erichson2020sparse}. The deflation approach works best if the eigenvectors of interest correspond to eigenvalues that are all separated by large gaps compared to the remaining eigenvalues. If several leading eigenvectors are desired and their corresponding eigenvalues are clustered, the block approach is preferable since it recovers a principal subspace instead of identifying individual eigenvectors separately.

For the single vector case, power iteration-based methods iterate the following steps:

\begin{enumerate}
    \item Update the current vector: $\Tilde{v}_{t+1}=Av_t$.
    \item Truncate or threshold the vector based on some penalty, usually $\ell_1$ or $\ell_0$.
    \item Re-normalize the vector: $v_{t+1}=\frac{\hat{v}_{t+1}}{\|\hat{v}_{t+1}\|}$.
\end{enumerate}

To recover several eigenvectors at once, one natural extension of the truncated power method is the truncated orthogonal iteration, see e.g. ITSPCA~\cite{ma2013sparse}. However, it cannot enforce orthogonality and sparsity at the same time. Performing an orthogonalization step (by computing a QR factorization or a singular value decomposition)  
is crucial to the algorithm stability, but destroys the sparsity pattern. On the other hand, performing truncation afterwards gives a sparse solution, but loses orthogonality. We propose the following framework based on the orthogonal iteration, where a post-processing step may be used to enforce sparsity. 

\begin{enumerate}
    \item Update the current vectors: $Q'_{t+1}=AQ_t$.
    \item Truncate or threshold (usually process each column of $Q'_{t+1}$ separately).
    \item Re-orthogonalize: QR: $Q_{t+1}=\textbf{qr}(Q'_{t+1})$. SVD: $U, S, V=\textbf{svd}(Q'_{t+1})$, $Q_{t+1}=UV^T$.
    \item (Optional) Post-processing: in each iteration, truncate or threshold $Q_{t+1}$.  
\end{enumerate}

We propose two variations of the Truncated Orthogonal Iteration in this framework. The first approach, formalized in~\cref{alg:TOrth}, is similar to the ITSPCA algorithm~\cite{ma2013sparse}, but with key differences in implementation and analysis. First, we replace the thresholding step with truncation, as discussed in~\Cref{sec: relations}. Second, we give a deterministic numerical analysis in~\cref{thm:thm1}, while for ITSPCA, \cite{ma2013sparse, cai2020optimal} established statistical convergence analyses under the spiked covariance model~\cite{johnstone2001distribution} and did not analyze the numerical convergence of the algorithm. Third, we use a different initialization scheme, i.e. warm initialization as discussed in~\Cref{sec:experiments}, as opposed to the ``diagonal thresholding'' initialization~\cite{johnstone2009consistency}, which also relies on the spiked covariance model and requires extra parameters. Our second approach, formalized in~\cref{alg:OrthT}, uses a greedy approach to get sparse vectors after each iteration of \cref{alg:TOrth}. 


\begin{algorithm}[H]
\caption{Truncated Orthogonal Iteration (TOrth)}\label{alg:TOrth}
\begin{algorithmic}
\STATE {Input: Symmetric positive semidefinite matrix $A\in\mathbb{S}^{p}$, initial vectors $Q_0\in \mathbb{R}^{p\times m}$}
\STATE {Output: An orthogonal (but possibly dense) matrix $Q_t$}
\STATE {Parameters: Cardinalities for each column vector $K=[k_1, \cdots, k_m]$}
\REPEAT
\STATE{Compute $P_t=AQ_{t-1}$}. Denote the $i^{th}$ column of $P_t$ as $p_i$.
\FOR{$i=1, \cdots, m$}
    \STATE{Let ${F}_i = \text{supp}(p_i, k_i)$ be the indices of $p_i$ with the largest $k_i$ absolute values.}
    \STATE{Compute $\hat{p}_i = \text{Truncate}(p_i, {F}_i)$.}
    \STATE{$\hat{P}_t[:,i]=\hat{p}_i$.}
\ENDFOR
\STATE{Reorthogonalize $Q_t = \textbf{qr}(\hat{P}_t)$.}
\STATE{$t\leftarrow t+1$}
\UNTIL{Convergence}
\end{algorithmic}
\end{algorithm}

\textbf{Complexity Analysis.}
For sparse PCA problems, suppose that we are given a data matrix $X\in \mathbb{R}^{n\times p}$. The covariance matrix is calculated by $\Sigma=X^TX$. In high-dimensional settings, $p\gg n\geq m$. In each iteration, matrix-matrix multiplication is $\mathcal{O}(npm)$. Sorting $m$ vectors of length $p$ in order to identify the largest entries is $\mathcal{O}(mp\log p)$. Since the dimension of matrix $\hat{P}_t$ is $p\times m$, QR factorization is $\mathcal{O}(pm^2)$. Compared to the single vector case, QR factorization is more expensive than normalizing $m$ single vectors (requiring $\mathcal{O}(mp)$ operations), but the deflation step can be avoided (saving $\mathcal{O}(np)$ operations). 

\section{Analysis}
\label{sec:convergence}

In this section, we analyze what happens in each iteration of \cref{alg:TOrth} when the matrix is perturbed and a truncation step is performed.

\subsection{Preliminaries.} 
We use the standard $\sin\Theta$ definition \cite{stewart1990matrix, van1983matrix} to measure the distance between subspaces:
\begin{definition}
Let $\mathcal{X}$ and $\mathcal{Y}$ be two $m$-dimensional subspaces of $\mathbb{R}^p$. Let the columns of $X$ form an orthonormal basis for $\mathcal{X}$ and the columns of $Y$ form an orthonormal basis for $\mathcal{Y}$.
We use $\|\sin\Theta(\mathcal{X},\mathcal{Y})\|_F$ to measure the distance between  
$\mathcal{X}$ and $\mathcal{Y}$,
where
\begin{equation}\label{eq:mat-angles-XY}
\Theta(\mathcal{X},\mathcal{Y})=\text{diag}(\theta_1(\mathcal{X},\mathcal{Y}),\ldots,\theta_m(\mathcal{X},\mathcal{Y})).
\end{equation}
Here, $\theta_j(\mathcal{X},\mathcal{Y})$'s denote the {\em canonical angles} between $\mathcal{X}$ and $\mathcal{Y}$ [p. 43]\cite{stewart1990matrix},
which is defined as
\begin{equation}\label{eq:indv-angles-XY}
0\le\theta_j(\mathcal{X},\mathcal{Y})\triangleq\arccos\sigma_j\le\frac {\pi}2\quad\mbox{for $1\le j\le m$},
\end{equation}
where $\sigma_j$'s are the singular values of $X^T Y$. 
Note that this definition is independent of which orthonormal bases $X$ and $Y$ are chosen for the spaces $\mathcal{X}$ and $\mathcal{Y}$.
\end{definition}
For the ease of notation, we use $\Theta(X,Y)=\Theta(\mathcal{X}, \mathcal{Y})$, where $X$, $Y$ are the orthonormal bases for the subspaces $\mathcal{X}$, $\mathcal{Y}$, respectively. It has been shown~\cite{stewart1990matrix, van1983matrix} that the following relations holds:
\begin{align}
    &\|\cos\Theta (X, Y)\|_{ui} = \|X^TY\|_{ui},\\
    &\|\sin\Theta (X, Y)\|_{ui} =\|{X^{\perp}}^TY\|_{ui}=\|X^TY^{\perp}\|_{ui},\\
    &\|\sin\Theta(X,Y)\|_2=\|XX^T-YY^T\|_2,\\
    &\|\sin\Theta(X,Y)\|_F=\frac{1}{\sqrt{2}}\|XX^T-YY^T\|_F=\sqrt{p-\|X^TY\|_F^2}.
\end{align}
where $\|\cdot\|_{ui}$ denotes any unitary invariant norm such as the 2-norm and the Frobenius norm. $X^{\perp}$ denotes the orthogonal complement of $X$.

Throughout the paper, we use the following well-known properties of matrix norms:
\begin{align}
    & \|A\|_2\leq \|A\|_F\leq \text{rank}(A)\|A\|_2, \label{eq:F2norm_relations}\\
    & \|AB\|_2\leq \|A\|_2 \|B\|_2,\\
    & \|AB\|_F \leq \|A\|_F\|B\|_2. \label{eq:Fnorm_ineq}
\end{align}

\subsection{Convergence Analysis}
We now establish our main result and key consequences.
\begin{theorem}\label{thm:thm1}
 Let $P$ be the matrix of eigenvectors corresponding to the $m$ largest eigenvalues of $\bar{A}$, with $\lambda_1(\bar{A})\geq \lambda_2(\bar{A})\geq \cdots \geq \lambda_m(\bar{A})>\lambda_{m+1}(\bar{A})>0$. Let $A = \bar{A} + E$. Assume $\lambda_1(\bar{A})=1$. Define $\gamma :=\frac{\lambda_{m+1}}{\lambda_{m}}<1$. Let $Q_t$ be the matrix obtained at iteration $t$ by \cref{alg:TOrth}. Then 
 \begin{equation}
     \|P^TQ_t\|_2^2 \geq\frac{\|P^TQ_{t-1}\|_F^2}{(1-\gamma^2)\|P^TQ_{t-1}\|_F^2+m\gamma^2} - \delta_E-\delta_{\text{Truncate}},
 \end{equation}
 where
 \[\delta_E = \frac{4\rho(E,K)}{\lambda_{m}^2(1-\|\sin\Theta(P, Q_{t-1})\|_2^2)}, \quad \delta_{\text{Truncate}}=2m\sqrt{\frac{\min\{\bar{k}_{\max}, p-k_{\min}\}}{p}}.\]
 \[\rho(E, K)=\max_{Q^TQ=I_m}\|EQ\|_2 \text{ subject to }\|q_i\|_0\leq k_i, \]
 \[\bar{k}_{\max}=\max_i\{\bar{k}_i\}=\max_i\{\|p_i\|_0\}, \ k_{\min}=\min_i\{k_i\}. \]
Assume that $\|P^TQ_t\|_F^2=c\|P^TQ_t\|_2^2$, where $c \in [1,m]$. Then we have:
\begin{equation}
    \|\sin\Theta (P, Q_{t})\|_F^2 \leq \frac{\gamma^2\|\sin\Theta (P, Q_{t-1})\|_F^2+\frac{m-c}{m}\|P^TQ_{t-1}\|_F^2}{1-(1-\gamma^2)\|\sin\Theta(P, Q_{t-1})\|_2^2} + c\delta_E + c\delta_{\text{Truncate}}.
\end{equation}
\end{theorem}

When $c\approx m$, we have:
\begin{equation}\label{eq: sinTheta}
    \|\sin\Theta(P, Q_t)\|_F^2\lessapprox \frac{\gamma^2\|\sin\Theta (P, Q_{t-1})\|_F^2}{1-(1-\gamma^2)\|\sin\Theta(P, Q_{t-1})\|_2^2} + m\delta_E + m\delta_{\text{Truncate}}.
\end{equation}
For $m=1$, the inequality~\cref{eq: sinTheta} holds exactly and we can derive a uniform convergence bound for $\|\sin\Theta(P, Q_t)\|_F^2$. For $m>1$, as $Q_t\rightarrow P$, $c\approx m$ and $\|\sin\Theta(P, Q_t)\|_F^2$ converges at the asymptotic rate $\gamma = \frac{\lambda_{m+1}}{\lambda_m}$.
\begin{remark}
A natural question that arises is whether truncating the vector only at the last step would give a better result. Besides computational concerns, truncating at each step helps to reduce the perturbation error, which is proportional to $\rho(E,k)$. At each step, if we do not truncate, i.e. $k=p$, then there is no truncation error, but $\rho(E,k)=\rho(E)$ can be large. On the other hand, if we truncate to $k$ nonzeros with $k\approx \bar{k}\ll p$, then the truncation error could potentially be large but $\rho(E,k)\approx\rho(E,\bar{k})\ll \rho(E)$. We recommend keeping $k$ close to $p$ in the first few iterations to avoid truncating the true nonzeros. At later steps, when the nonzero indices of $x_t$ include the nonzero indices of $\bar{x}$, it is safe to truncate to a smaller $k$ without much truncation error and the perturbation error is kept low at the same time. 
\end{remark}

To prove \cref{thm:thm1}, we need the following lemmas: \cref{lemma:orth} measures the progress made by each standard orthogonal iteration, \cref{lemma:error} accounts for the perturbation error, and \cref{lemma:truncate} analyzes the truncation step. 

We first measure the progress made by the standard orthogonal iteration without any truncation or perturbation. In~\cite{van1983matrix} it has been shown that the distance between the $t^{th}$ updated matrix $Q_t\in\mathbb{R}^{p\times m}$ and the matrix of first $m$ eigenvectors $P$ converges at a rate $\gamma=|\lambda_{m+1}/\lambda_{m}|$:
\begin{equation}\label{eq:twonorm}
    \|\sin \Theta (Q_t, P)\|_2\leq \gamma^t \frac{\|\sin \Theta (Q_0, P)\|_2}{\sqrt{1-\|\sin \Theta (Q_0, P)\|_2^2}} .
\end{equation}
When $m=1$, define $\theta_t\in[0, \pi/2]$ by $\cos(\theta_t)=|p^T q_t|$ and this reduces to 
\begin{equation}
    \sin \theta_t \leq \gamma^t \tan \theta_0 .
\end{equation}
An equivalent bound measured in Frobenius norm can be derived from~\cite{van1983matrix} (the proof can be found in the Appendix):
\begin{equation}\label{eq:Fnorm}
    \|\sin \Theta (Q_t, P)\|_F\leq \gamma^t \frac{\|\sin \Theta (Q_0, P)\|_F}{\sqrt{1-\|\sin \Theta (Q_0, P)\|_2^2}}.
\end{equation}
A one-step bound can also be derived from~\cite{van1983matrix}:
\begin{equation}\label{eq:onestep}
    \|\sin \Theta (Q_{t}, P)\|_F\leq \gamma \frac{\|\sin \Theta (Q_{t-1}, P)\|_F}{\sqrt{1-\|\sin \Theta (Q_{t-1}, P)\|_2^2}}.
\end{equation}
We provide the following lemma for a similar approximation of the distance update in each iteration:

\begin{lemma}\label{lemma:orth}
Let $P$ be the matrix of eigenvectors corresponding to the largest $m$ (in absolute value) eigenvalues of a symmetric matrix $\bar{A}$, and $\Lambda_m=\diag(\lambda_1, \cdots, \lambda_m)$, and let $\gamma=|\lambda_{m+1}/\lambda_{m}|$. Given any $Q_{t-1}\in \mathbb{R}^{p\times m}$ such that $Q_{t-1}^T Q_{t-1}=I$, let $Q_t$ be the orthogonal matrix obtained by QR factorization of $\bar{A}Q_{t-1}$, i.e. $Q_tR_t=\bar{A} Q_{t-1}$, then
\begin{equation}\label{eq:standard}
    \|P^TQ_t\|^2_2\geq \frac{\|P^TQ_{t-1}\|^2_F}{(1-\gamma^2)\|P^TQ_{t-1}\|^2_F+m\gamma^2}.
\end{equation}
Assume that $\|P^TQ_t\|_F^2=c\|P^TQ_t\|_2^2$, where $c \in [1,m]$. Then we have:
\begin{equation}
\|\sin\Theta (P, Q_{t})\|_F^2 \leq \frac{\gamma^2\|\sin\Theta (P, Q_{t-1})\|_F^2+\frac{m-c}{m}\|P^TQ_{t-1}\|_F^2}{1-(1-\gamma^2)\|\sin\Theta(P, Q_{t-1})\|_2^2}.
\end{equation}
\end{lemma} 

\begin{proof}
We can decompose $Q_{t-1}$ as $Q_{t-1}=PX+P^{\perp}Y$, where $P^{\perp}$ is the orthogonal complement of $P$ and its columns are the eigenvectors of $\bar{A}$ corresponding to the $m+1, \cdots p$ eigenvalues, i.e. $\bar{A}P=P\Lambda_m$, $\bar{A}P^{\perp}=P^{\perp}\Lambda'$, where $\Lambda'=\diag(\lambda_{m+1}, \cdots, \lambda_p)$. We have the following equations:
\begin{align}
    & Q_{t-1}^TQ_{t-1}=X^TX+Y^TY=I \Rightarrow \|X\|_F^2+\|Y\|_F^2=m, \label{eq:e313}\\
    & \|P^TQ_tR_t\|_F^2=\|P^T\bar{A}Q_{t-1}\|_F^2=\|P^T\bar{A}(PX+P^{\perp}Y)\|_F^2=\|\Lambda_m X\|_F^2\geq \lambda_m^2\|X\|_F^2. \label{eq:e314}
\end{align}

Since $\|Q_tR_t\|=\|R_t\|=\|\bar{A}Q_{t-1}\|$, and
\begin{equation}\label{eq:e315}
    \|\bar{A}Q_{t-1}\|_F^2=\|P\Lambda_m X +P^{\perp} \Lambda' Y\|_F^2=\|\Lambda_m X\|_F^2 +\| \Lambda' Y\|_F^2,
\end{equation}
We have
\begin{align}
    \|P^TQ_t\|_2^2 &\geq \frac{\|P^T Q_tR_t\|_F^2}{\|R_t\|_F^2}=\frac{\| P^T Q_tR_t\|_F^2}{\|\bar{A}Q_{t-1}\|_F^2} \text{    by~\cref{eq:Fnorm_ineq}}\\
    &\geq \frac{\|\Lambda_m X\|_F^2}{\|\Lambda_m X\|_F^2 +\| \Lambda' Y\|_F^2} \text{   by~\cref{eq:e314} and \cref{eq:e315}}\\
    &\geq \frac{\lambda_m^2\|X\|_F^2}{\lambda_m^2\|X\|_F^2+\lambda_{m+1}^2\|Y\|_F^2}  \text{   by~\cref{eq:e314}}\\
    &=\frac{\lambda_m^2\|X\|_F^2}{\lambda_m^2\|X\|_F^2+\lambda_{m+1}^2(m-\|X\|_F^2)}   \text{   by~\cref{eq:e313}}\\
    &=\frac{\|P^TQ_{t-1}\|^2_F}{(1-\gamma^2)\|P^TQ_{t-1}\|^2_F+m\gamma^2}.\label{eq:PQt_result}
\end{align}

Assume that $\|P^TQ_t\|_F^2=c\|P^TQ_t\|_2^2$, where $c \in [1,m]$, then we have
\begin{align}
    \|\sin\Theta(P, Q_t)\|_F^2 &= m-\|P^TQ_t\|_F^2 = m-c\|P^TQ_t\|_2^2\\
    &\leq \frac{m\gamma^2\|\sin\Theta (P, Q_{t-1})\|_F^2+(m-c)\|P^TQ_{t-1}\|_F^2}{m-(1-\gamma^2)\|\sin\Theta(P, Q_{t-1})\|_F^2} \text{   by~\cref{eq:PQt_result}}\\
    &\leq \frac{\gamma^2\|\sin\Theta (P,Q_{t-1})\|_F^2+\frac{m-c}{m}\|P^TQ_{t-1}\|_F^2}{1-(1-\gamma^2)\|\sin\Theta(P, Q_{t-1})\|_2^2}
    \text{   by~\cref{eq:F2norm_relations}.}
\end{align}

When $c\approx m$, 
\begin{equation}
    \|\sin\Theta(P, Q_t)\|_F^2\lessapprox \frac{\gamma^2\|\sin\Theta (P, Q_{t-1})\|_F^2}{1-(1-\gamma^2)\|\sin\Theta(P, Q_{t-1})\|_2^2}.
\end{equation}

\end{proof}

\cref{lemma:orth} measures the progress made by the orthogonalization step, and the QR factorization can be replaced by other factorization methods, as long as the updated matrix $Q$ is orthogonal.



\begin{lemma}\label{lemma:error}
Suppose that $A=\bar{A}+E$, where $A$ and $\bar{A}$ are symmetric positive semidefinite matrices. Suppose that $P$ is the matrix of eigenvectors of $\bar{A}$ corresponding to the largest $m$ eigenvalues. Assume $\lambda_1(\bar{A})=1$. Let $Q\in\mathbb{R}^{p\times m}$ be any sparse matrix such that $Q^TQ=I$, $\|q_i\|_0\leq k_i$. Then 

\begin{equation}
    \frac{\|P^TAQ\|_F^2}{\|AQ\|_F^2}\geq \frac{\|P^T\bar{A}Q\|_F^2}{\|\bar{A}Q\|_F^2}-\frac{4m\rho(E,K)}{\|\bar{A}Q\|_F^2},
\end{equation}
where $\rho(E, K)=\max_{Q^TQ=I_m}\|EQ\|_2 \text{ subject to }\|q_i\|_0\leq k_i, \ K=[k_1,\cdots, k_m]$.
\end{lemma}

\begin{proof}
Since $A=\bar{A}+E$, using triangle inequality of norms, we have
\begin{align}
    \frac{\|P^TAQ\|_F}{\|AQ\|_F} &= \frac{\|P^T(\bar{A}+E)Q\|_F}{\|(\bar{A}+E)Q\|_F} \geq \frac{\|P^T\bar{A}Q\|_F-\|P^TEQ\|_F}{\|\bar{A}Q\|_F+\|EQ\|_F}\\
    &\geq \frac{\|P^T\bar{A}Q\|_F-\|P^TEQ\|_F-\|EQ\|_F}{\|\bar{A}Q\|_F}\\
    &\geq \frac{\|P^T\bar{A}Q\|_F-\sqrt{m}(\|P^TEQ\|_2+\|EQ\|_2)}{\|\bar{A}Q\|_F}\\
    &\geq \frac{\|P^T\bar{A}Q\|_F-2\sqrt{m}\rho(E,K)}{\|\bar{A}Q\|_F}\\
    \Rightarrow \frac{\| P^T AQ \|_F^2}{\| AQ \|_F^2} &\geq \frac{\| P^T \bar{A} Q \|_F^2 - 4 \sqrt{m} \| P^T \bar{A} Q \|_F\rho(E,K) + 4m \rho(E,K)^2}{\| \bar{A}Q \|_F^2}\\
    &\geq \frac{\|P^T\bar{A}Q\|_F^2}{\|\bar{A}Q\|_F^2}-\frac{4m\rho(E,K)}{\|\bar{A}Q\|_F^2}.
\end{align}

\end{proof}

\begin{remark}
 In \cref{alg:TOrth}, $Q_tR_t=\text{Truncate}(AQ_{t-1})$, and in theory $Q_t$ can be dense. If the columns of 
 $\text{Truncate}(A Q_{t-1})$ mostly have nonzeros in nonoverlapping sets of rows, then
 the columns of this matrix will be almost orthogonal, and $Q_t$ will be similarly
 sparse. The first column of $Q_t$ will definitely 
 have the same sparsity as the first column of $\text{Truncate}(A Q_{t-1})$, but later 
 columns are orthogonalized against more vectors and so may become denser.
 \end{remark}

To measure the loss incurred during truncation, we establish the lemma below. 

\begin{lemma}\label{lemma:truncate}
Consider a unit vector $\bar{x}\in\mathbb{R}^p$ with support set $supp(\bar{x})=\bar{F}$, and $\bar{k}=|\bar{F}|$. Consider a vector $y$ with the $k-$largest absolute values with indices in set $F$. Then
\begin{equation}
    \frac{|\text{Truncate}(y, F)^T\bar{x}|}{\|\text{Truncate}(y,F)\|}\geq \frac{|y^T\bar{x}|}{\|y\|}-\sqrt{\frac{\min\{\bar{k}, p-k\}}{p}}.
\end{equation}
\end{lemma}

\begin{proof}
Let $F_1 =\bar{F}\symbol{92} F$, $F_2 = \bar{F}\cap F$, $F_3= F\symbol{92} \bar{F}$, then
\[\bar{x}=\bar{x}_{F_1}+\bar{x}_{F_2},\ \text{Truncate}(y, F)=y_{F_2}+y_{F_3}.\] 
Let $\bar{\alpha}=\|\bar{x}_{F_1}\|\leq 1$, $\alpha=\|y_{F_1}\|$, $k_1=|F_1|$. Note that if $\bar{F}\subseteq F$, then $F_1=\emptyset$ and $k_1=0$, $\bar{\alpha}=0$. If $k_1\neq 0$ we have:
\begin{equation}
    \frac{\alpha^2}{k_1}=\frac{\sum_{i\in F_1}y_i^2}{|F_1|}\leq \frac{\|y\|^2}{p},
\end{equation}
since $y_{F_1}$ contains the $k1-$smallest entries in $y$.
We know $k_1=|\bar{F}\symbol{92} F|\leq \min\{\bar{k}, p-k\}$, therefore 
\begin{align}
    & \alpha\leq \sqrt{\frac{k_1}{p}}\|y\|\leq \min\{\sqrt{\frac{\bar{k}}{p}}, \sqrt{\frac{p-k}{p}}\}\|y\|,\\
    &  |\text{Truncate}(y, F)^T\bar{x}|\geq |y^T\bar{x}|- \bar{\alpha}\alpha\geq |y^T\bar{x}|- \alpha\\
    & \Rightarrow |\text{Truncate}(y, F)^T\bar{x}|\geq |y^T\bar{x}|-\min\{\sqrt{\frac{\bar{k}}{p}}, \sqrt{\frac{p-k}{p}}\}\|y\|.
\end{align}

\end{proof}


\noindent\textbf{Truncation error of the matrix product $P^TQ$.} We denote each column of $Q$ by $q_i$ and each column of $P$ by $p_j$. Let $|F_i|=k_i$ and $\|p_j\|_0=\bar{k}_j$. Then the truncation error of the matrix product $P^TQ$ is given by:
\begin{align}
    (\text{Truncate}(q_i, F_i)^Tp_j)^2 &\geq (q_i^Tp_j)^2-2\sqrt{\frac{\min\{\bar{k}_j, p-k_i\}}{p}}\|q_i\|^2,\\
    \sum_{i,j}(\text{Truncate}(q_i, F)^Tp_j)^2 &\geq \sum_{i,j}(q_i^Tp_j)^2-2\sum_{i, j}\sqrt{\frac{\min\{\bar{k}_j, p-k_i\}}{p}} \|q_i\|^2
\end{align}
\begin{equation}\label{eq:truncate}
   \Rightarrow \frac{\|\text{Truncate}(Q)^TP\|_F^2}{\|\text{Truncate}(Q)\|_F^2}\geq \frac{\|Q^TP\|_F^2}{\|Q\|_F^2}-2m\sqrt{\frac{\min\{\bar{k}_{\max}, p-k_{\min}\}}{p}},
\end{equation}
where $\bar{k}_{\max}=\max_j\{\bar{k}_j\}$ and $k_{\min}=\min_i \{k_i\}$.

\begin{remark}
This bound is not tight. Assume that: $\bar{F_i}\subseteq F_i^{(t)}, \ \forall i$, then there is no truncation error and \begin{equation}
    \frac{\|\text{Truncate}(Q)^TP\|_F^2}{\|\text{Truncate}(Q)\|_F^2}\geq \frac{\|Q^TP\|_F^2}{\|Q\|_F^2}.
\end{equation}
\end{remark}



Putting everything together, we can now prove \cref{thm:thm1}:
\begin{proof}
Based on~\cref{alg:TOrth}, $Q_t$ is obtained by doing \textbf{qr} factorization on $\text{Truncate}(AQ_{t-1})$, i.e.
\begin{equation}
    Q_tR_t = \text{Truncate}(AQ_{t-1}).
\end{equation}
\begin{align}
\|P^TQ_t\|_2^2 &\geq \frac{\|P^T\text{Truncate}(AQ_{t-1})\|_F^2}{\|R_t\|_F^2}= \frac{\|P^T\text{Truncate}(AQ_{t-1})\|_F^2}{\|\text{Truncate}(AQ_{t-1})\|_F^2} \text{  by~\cref{eq:Fnorm_ineq}}\\
&\geq \frac{\|P^TAQ_{t-1}\|_F^2}{\|AQ_{t-1}\|_F^2}-\delta_{\text{Truncate}} \text{  by~\cref{eq:truncate}}\\
&\geq \frac{\|P^T\bar{A}Q_{t-1}\|_F^2}{\|\bar{A}Q_{t-1}\|_F^2}- \frac{4m\rho(E,K)}{\|\bar{A}Q_{t-1}\|_F^2}-\delta_{\text{Truncate}} \text{   by~\cref{lemma:error}}\\
&\geq  \frac{\|P^T\bar{A}Q_{t-1}\|_F^2}{\|\bar{A}Q_{t-1}\|_F^2}- \frac{4\rho(E,K)}{\lambda_{m}^2(1-\|\sin\Theta(P, Q_{t-1})\|_2^2)} - \delta_{\text{Truncate}}\label{eq:e325}\\
&\geq \frac{\|P^TQ_{t-1}\|_F^2}{(1-\gamma^2)\|P^TQ_{t-1}\|_F^2+m\gamma^2}-\delta_E-\delta_{\text{Truncate}}  \text{   by~\cref{lemma:orth}}
\end{align}
The inequality~\cref{eq:e325} holds due to the fact that:
\[\frac{m}{\|\bar{A}Q_{t-1}\|_F^2}\leq \|(\bar{A}Q_{t-1})^{-1}\|_2^2\leq \frac{1}{\lambda_{m}^2(1-\|\sin\Theta(P, Q_{t-1})\|_2^2)} \text{  by~\cref{eq:RtNorm}}.\]
Assume that $\|P^TQ_t\|_F^2=c\|P^TQ_t\|_2^2$, where $c \in [1,m]$. Then we have:
\begin{align}
\|\sin\Theta(P, Q_t)\|_F^2 &= m-\|P^TQ_t\|_F^2 = m-c\|P^TQ_t\|_2^2\\
&\leq m - c\frac{\|P^TQ_{t-1}\|_F^2}{(1-\gamma^2)\|P^TQ_{t-1}\|_F^2+m\gamma^2} + c\delta_E + c\delta_{\text{Truncate}}\\
&= \frac{m\gamma^2\|\sin\Theta (P, Q_{t-1})\|_F^2+(m-c)\|P^TQ_{t-1}\|_F^2}{m-(1-\gamma^2)\|\sin\Theta(P, Q_{t-1})\|_F^2} + c\delta_E + c\delta_{\text{Truncate}}\\
&\leq \frac{\gamma^2\|\sin\Theta (P, Q_{t-1})\|_F^2+\frac{m-c}{m}\|P^TQ_{t-1}\|_F^2}{1-(1-\gamma^2)\|\sin\Theta(P, Q_{t-1})\|_2^2} + c\delta_E + c\delta_{\text{Truncate}}
\end{align}
When $\|\sin\Theta(P, Q_{t-1})\|_F^2 \geq \frac{m-c}{1-\gamma^2}$, 
\begin{equation}
    \frac{m\gamma^2\|\sin\Theta (P, Q_{t-1})\|_F^2+(m-c)\|P^TQ_{t-1}\|_F^2}{m-(1-\gamma^2)\|\sin\Theta(P, Q_{t-1})\|_F^2}\leq \|\sin\Theta(P, Q_{t-1})\|_F^2.
\end{equation}
When $c\approx m$, 
\begin{equation}
    \|\sin\Theta(P, Q_t)\|_F^2\lessapprox \frac{\gamma^2\|\sin\Theta (P, Q_{t-1})\|_F^2}{1-(1-\gamma^2)\|\sin\Theta(P, Q_{t-1})\|_2^2} + m\delta_E + m\delta_{\text{Truncate}}.
\end{equation}
\end{proof}

We have proved that the orthogonal projector $\mathcal{Q}^{(t)}=Q_tQ_t^T$ onto $\text{span}\{Q_t\}$ converges to the true orthogonal projector $\mathcal{P}=PP^T$ onto $\text{span}\{P\}$. We now show that when the sequence of projectors converges, each column vector of $Q_t$ also converges to the corresponding column vector of $P$. Let $[Q]_i$ be the first $i$ columns of $Q$. Let $\mathcal{Q}_{i}^{(t)}:= [Q]_i[Q]_i^T$ be the orthogonal projector onto $\text{span}\{[Q_t]_i\}$ and $\mathcal{P}_{i}$ the orthogonal projector onto $\text{span}\{[P]_i\}$. Then we have the following corollary:

\begin{corollary}
If the sequence of orthogonal projectors $\{\mathcal{Q}_{i}^{(t)}\}_{t\geq 1}$ converges to $\mathcal{P}_{i}$ for all $i$, then each column vector of $Q_t$ converges to the corresponding column vector of $P$.
\end{corollary}

\begin{proof}
We can express the ${i+1}^{th}$ column vector of $Q_t$ as
\[q_{i+1}^{(t)}=(\mathcal{Q}_{i+1}^{(t)}-\mathcal{Q}_{i}^{(t)}) q_{i+1}^{(t)}\]
\[=[(\mathcal{Q}_{i+1}^{(t)}-\mathcal{P}_{i+1})+(\mathcal{P}_{i+1}-\mathcal{P}_{i})+(\mathcal{P}_{i}-\mathcal{Q}_{i}^{(t)})]q_{i+1}^{(t)}\]
\[=[(\mathcal{Q}_{i+1}^{(t)}-\mathcal{P}_{i+1})+(\mathcal{P}_{i}-\mathcal{Q}_{i}^{(t)})]q_{i+1}^{(t)} + p_{i+1}p_{i+1}^T q_{i+1}^{(t)}\]
\[\Rightarrow 1 - |p_{i+1}^T q_{i+1}^{(t)}|\leq \|(\mathcal{Q}_{i+1}^{(t)}-\mathcal{P}_{i+1})+(\mathcal{P}_{i}-\mathcal{Q}_{i}^{(t)})\|_2. \]
Since the projector $\mathcal{Q}_{i+1}^{(t)}$ converges to $\mathcal{P}_{i+1}$, and $\mathcal{Q}_{i}^{(t)}$ converges to $\mathcal{P}_{i}$, we have $q_{i+1}^{(t)}$ converges to $p_{i+1}$.
\end{proof}

\section{Truncated Orthogonal Iteration with Strict Sparsity Constraint}\label{sec:sparse}
In real scenarios, final output vectors are often required to be sparse for better interpretation. Sparsity constraints are not enforced in TOrth or in ITSPCA~\cite{ma2013sparse}, since by doing QR decomposition, the $m^{th}$ vector needs to be orthogonalized against all previously obtained vectors $v_1, \cdots, v_{m-1}$. If the previous vectors are partially overlapping, the $m^{th}$ vector is likely to be dense. Therefore we also present \cref{alg:OrthT} that performs a truncation step on the vector obtained by QR, which we denote as $Q$. Specifically, we solve the following problem:
\[
Q_{\text{truncate}}=\arg\min_{\hat{Q}} \|Q-\hat{Q}\|_F^2,\quad \text{subject to: }\|\hat{q}_i\|_0\leq k_i, \ \|\hat{q}_i\|_2=1.
\]

\begin{algorithm}[H]
\caption{Truncated Orthogonal Iteration with Post Truncation (TOrthT)}\label{alg:OrthT}
\begin{algorithmic}
\STATE {Input: Symmetric positive semidefinite matrix $A\in\mathbb{S}^{p}$, initial vector $Q_0\in \mathbb{R}^{p\times m}$, cardinality vector $K\in\mathbb{R}^m$}
\STATE {Output: A sparse and near-orthogonal matrix $Q_t$}
\REPEAT
\STATE{Compute $P_t=AQ_{t-1}$}. Denote the $i^{th}$ column of $P_t$ as $p_i$.
\FOR{$i=1, \cdots, m$}
    \STATE{Let ${F}_i = \text{supp}(p_i, k_i)$ be the indices of $p_i$ with the largest $k_i$ absolute values.}
    \STATE{Compute $\hat{p}_i = \text{Truncate}(p_i, {F}_i)$.}
    \STATE{$\hat{P}_t[:,i]=\hat{p}_i$}
\ENDFOR
\STATE{Reorthogonalize $\Tilde{Q}_t = \textbf{qr}(P_t)$. Denote the $i^{th}$ column of $\Tilde{Q}_t$ as $\Tilde{q}_i$.}
\FOR{$i=1, \cdots, m$}
    \STATE{Let ${F}_i = \text{supp}(\Tilde{q}_i, k_i)$ be the indices of $\Tilde{q}_i$ with the largest $k_i$ absolute values.}
    \STATE{Compute $\hat{q}_i = \text{Truncate}(\Tilde{q}_i, \mathcal{F}_i)$.}
    \STATE{Compute $q_i = \frac{\hat{q}_i}{\|\hat{q}_i\|}$.}
    \STATE{$Q_t[:,i]=q_i$}
\ENDFOR
\STATE{$t\leftarrow t+1$}
\UNTIL{Convergence}
\end{algorithmic}
\end{algorithm}

To test the impact of this post-truncation step, we construct a toy example where $\bar{A}=V\Lambda V^T$, $A = \bar{A}+E$ and $\rho(\bar{A})=1$, $\rho(E)\approx 0.21$. We apply \cref{alg:OrthT} to three simple scenarios: a) the leading eigenvectors that we are trying to recover have completely overlapping nonzero indices; b) the leading eigenvectors have partially overlapping nonzero indices; and c) the indices are non-overlapping. We keep $K=[p,p,p]=[100,100,100]$ in the first 20 iterations, $K=[50,50,50]$ in iteration $21-40$, $K=[25,25,25]$ in iteration $41-60$, and $K=[10,10,10]$ in iteration $61-80$. In all three cases, $\|Q_{\text{truncate}}-Q\|_F^2$ is kept relatively low ($\approx 1e-4$) when the algorithm converges, as shown in \cref{fig:threecases}. For the non-overlapping case, the algorithm is able to recover the true support set before truncation. The orthogonality loss of the final output matrix $Q_t$ is also calculated, and $\|I-Q_t^TQ_t\|_F^2=1.57e-4$ in case a), 1.17e-4 in case b) and 0 in case c).  
The post-truncation step does not have a significant impact on the quality or orthogonality of the obtained vectors, especially when the true eigenvectors have non-overlapping support sets.

\begin{figure}
     \centering
     \begin{subfigure}[b]{0.3\textwidth}
         \centering
         \includegraphics[width=\textwidth]{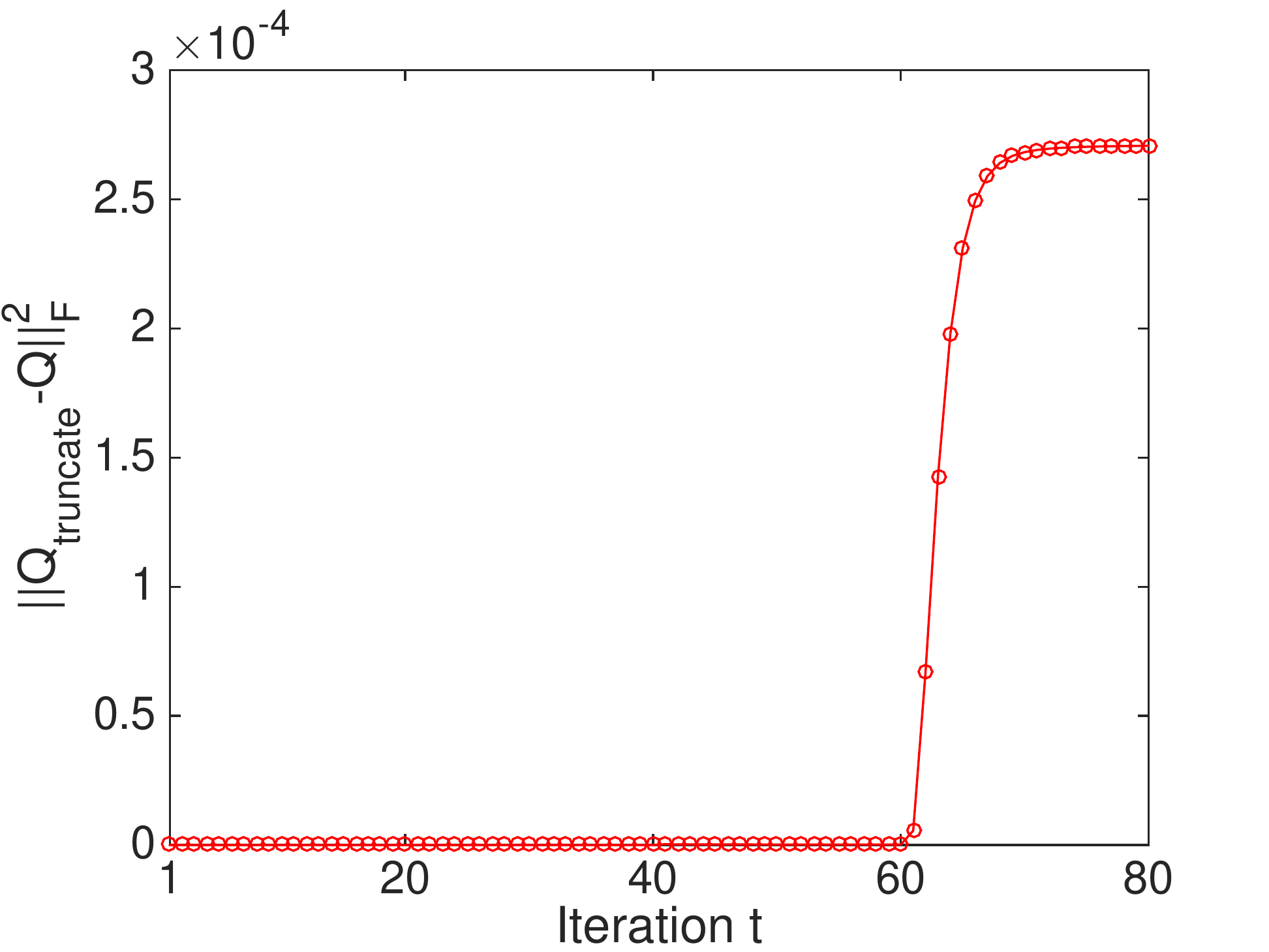}
         \caption{completely overlapping}
     \end{subfigure}
     \hfill
     \begin{subfigure}[b]{0.3\textwidth}
         \centering
         \includegraphics[width=\textwidth]{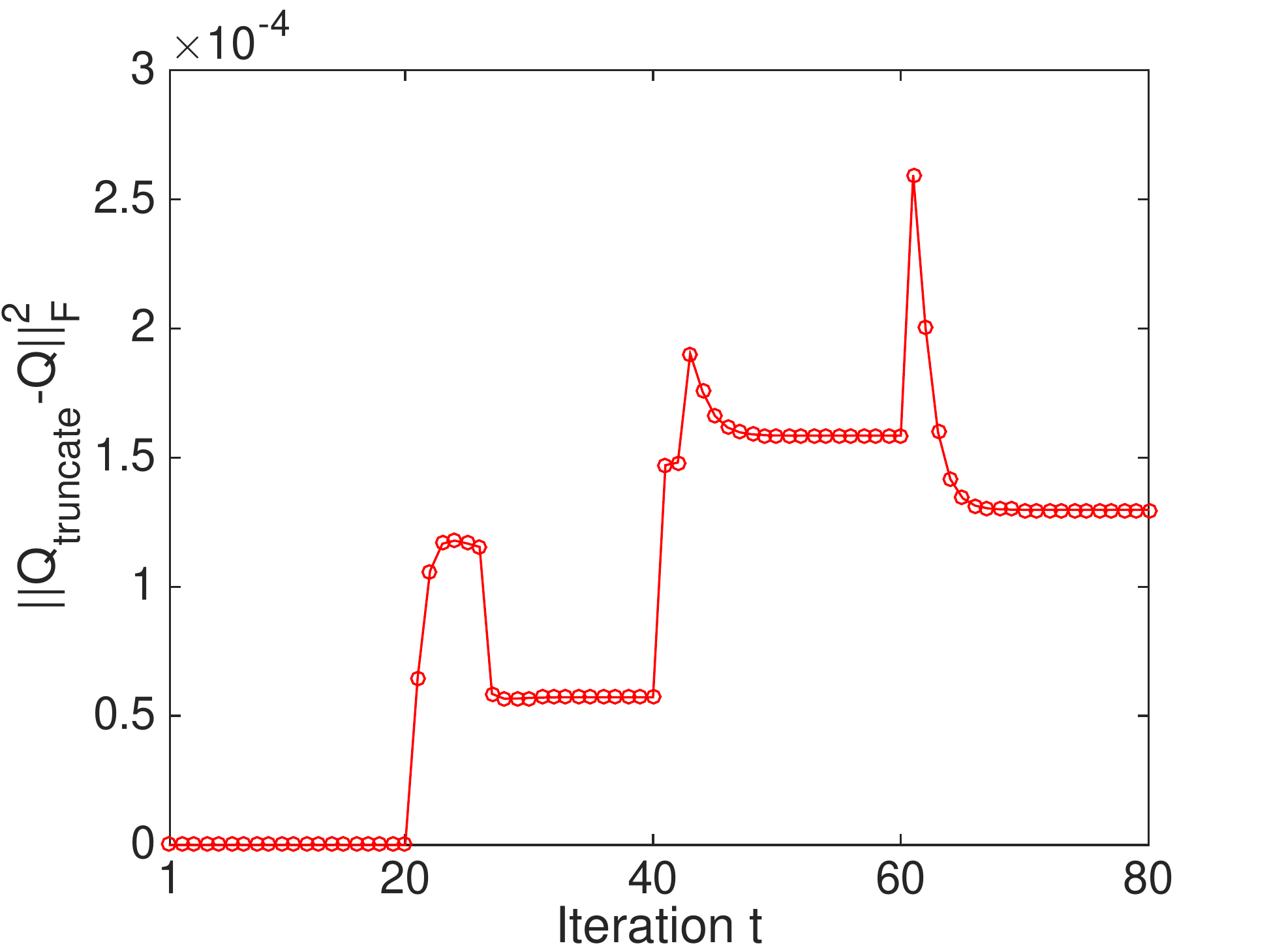}
         \caption{partially overlapping}
     \end{subfigure}
     \hfill
     \begin{subfigure}[b]{0.3\textwidth}
         \centering
         \includegraphics[width=\textwidth]{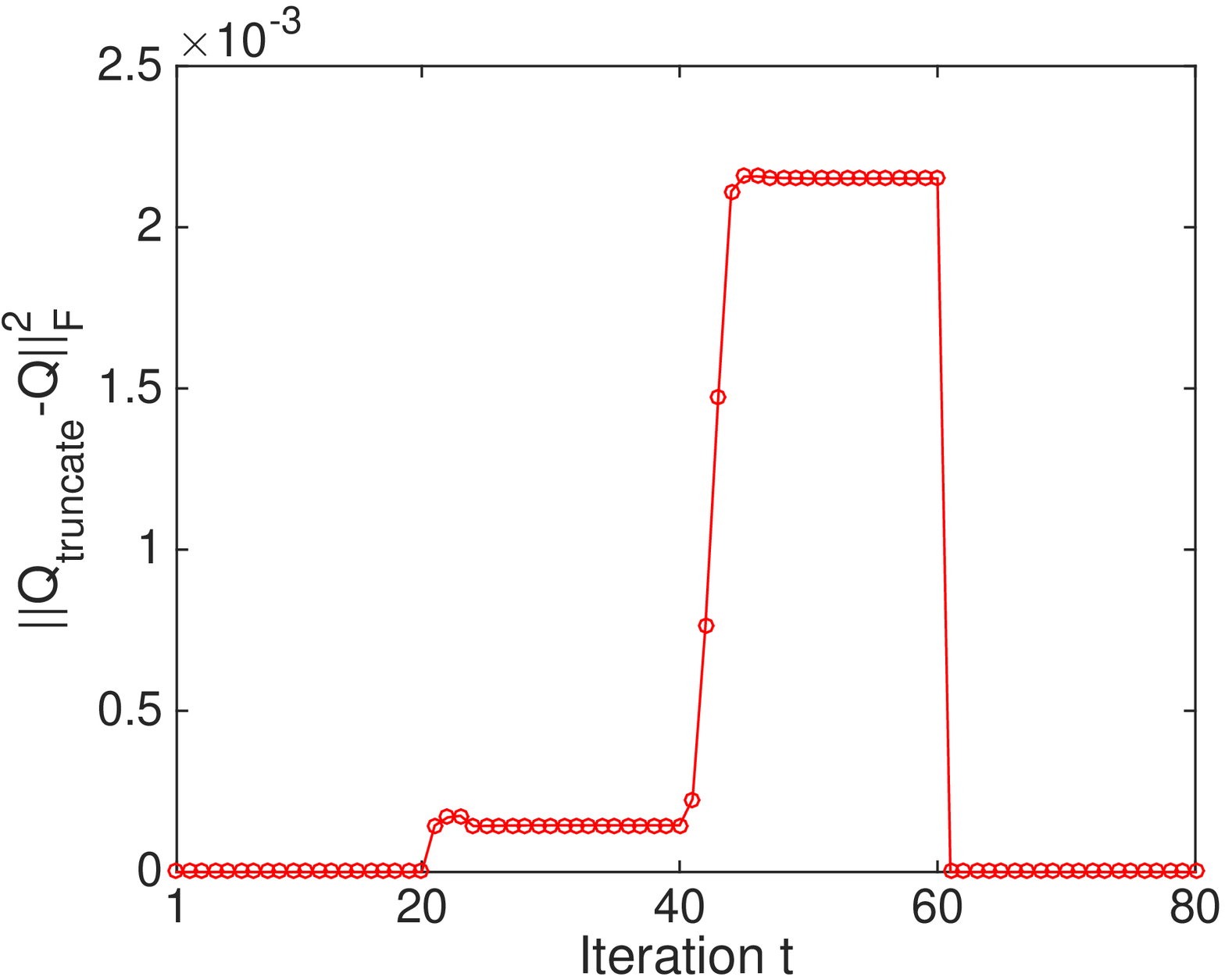}
         \caption{non-overlapping}
     \end{subfigure}
        \caption{(a) (b) (c) correspond to three different cases of support sets of the leading eigenvectors. We plot $\|Q_{\text{truncate}}-Q_t\|_F^2$ vs. iteration number $t$ to measure the difference between the output matrix before and after the truncation step. The truncation level starts at $p$ (nothing is truncated) and is decreased at iteration 21, 41, 61 and 81.}
        \label{fig:threecases}
\end{figure}

\section{Relations to Existing Algorithms}\label{sec: relations}
When estimating a single sparse eigenvector, \cref{thm:thm1} can be applied directly to the Truncated Power Method (TPower)~\cite{yuan2013truncated} using $m=1$. In this case, we have the following corollary:

\begin{corollary}
Let $P$ be the matrix of eigenvectors corresponding to the $m-$largest eigenvalues of $\bar{A}$. $A = \bar{A} + E$. Define $\gamma :=\frac{\lambda_{2}}{\lambda_{1}}<1$. Let $q_t$ be the matrix obtained at iteration $t$ by the Truncated Power method. Then
\[|\sin\angle(p, q_t)|^2\leq \frac{\gamma^2|\sin\angle (p, q_{t-1})|^2}{1-(1-\gamma^2)|\sin\angle(p, q_{t-1})|^2} + \delta_E + \delta_{\text{Truncate}}.\]
And a uniform bound is given by
\[|\sin\angle(p, q_t)|\leq \mu^t |\sin\angle(p, q_0)|+\sqrt{\delta_E+\delta_{\text{Truncate}}}\frac{1-\mu^t}{1-\mu},\]
where 
\[\mu = \frac{\gamma}{\sqrt{1-(1-\gamma^2)|\sin\angle(p, q_0)|^2}}\leq 1, \ \delta_E = \frac{4\rho(E,k)}{\lambda_2^2(1-|\sin\angle(p, q_0)|^2)}, \ \delta_{\text{Truncate}} = 2\sqrt{\frac{\min\{\bar{k}, p-k\}}{p}}.\]
\end{corollary}

The Iterative Thresholding Sparse PCA (ITSPCA) proposed by ~\cite{ma2013sparse} also fits into the framework introduced in \Cref{sec:method} and uses a thresholding step. The thresholding step is performed through a user-specified function $\eta$ which satisfies:
\begin{equation}\label{eq:eta}
    |\eta(y_i, t)-y_i|\leq t, \ \eta(y_i, t)\textbf{1}_{|y_i|<t}=0\ \forall y_i\in y, \ t>0. 
\end{equation}
Common thresholding techniques include soft and hard thresholding, as well as a range of operators in between, see e.g. SCAD~\cite{fan2001variable}. The analysis in~\cref{thm:thm1} can be applied to ITSPCA if we substitute~\cref{lemma:truncate} by the following lemma: 

\begin{lemma}\label{lemma:threshold}
Consider a unit vector $\bar{x}\in\mathbb{R}^p$ with support set $supp(\bar{x})=\bar{F}$, and $\bar{k}=|\bar{F}|$. Consider a vector $y$ that is thresholded by the user-specified thresholding function $\eta$, where $\eta$ satisfies~\cref{eq:eta}.Then
\begin{equation}
    \frac{\text{Thresh}(y, t)^T\bar{x}}{\|\text{Thresh}(y, t)\|_2}\geq \frac{|y^T\bar{x}|}{\|y\|_2}-\frac{t\sqrt{\bar{k}}}{\|y\|_2}.
\end{equation}
\end{lemma}

Analysis in~\cite{ma2013sparse} under the spiked covariance model is also applicable to our \cref{alg:TOrth} since we can transform our sparsity constraint, which is an upper bound on $\ell_0$ norm, to an upper bound on $\ell_1$ norm:
\[\|q_j\|_0\leq k_j, \ \|q_j\|_2=1 \Rightarrow \|q_j\|_1\leq \sqrt{k_j}.\]

\section{Experimental results}
\label{sec:experiments}
In this section we show several numerical experiments to demonstrate the efficiency and accuracy of the proposed algorithms. Because the Truncated Orthogonal Iteration is a direct extension of the Truncated Power Method to recover multiple units, we first apply both methods on a simulated dataset with true sparse eigenvectors and report their runtime and distance to the true eigenvectors. We see that the Truncated Orthogonal Iteration is less sensitive to random initialization compared to TPower. We then use the PitProps Dataset, a standard benchmark, to evaluate the performance of the proposed algorithms and compare with state of the art. We also apply our algorithm to a sea surface temperature dataset for recovery of sparse patterns, and to the MNIST dataset for classification of handwritten digits. We refer to \cref{alg:TOrth} as TOrth and \cref{alg:OrthT} as TOrthT.

\textbf{Initialization.} We use the same warm initialization strategy as specified in \cite{yuan2013truncated}, i.e. starting with a larger $k$ and use the output as the initialization for a smaller $k$. Specifically, for each column vector we run the algorithm with $\{8k_i, 4k_i, 2k_i, k_i\}$ sequentially and use the original dimension $p$ if the multiple of $k$ is greater than $p$.

\textbf{Choosing the Cardinality Parameter $k$.} Unless the cardinality parameter $k$ is pre-specified, we start with a large cardinality constraint and gradually tighten it. Once the subspace converges for $k$, we proceed to the next truncation level with cardinality parameter $k/2$. If $k$ is too small, $\|\sin\Theta(Q_t, Q_{t-1}\|_F^2)$ is likely to have a big jump, as illustrated in \cref{fig:init}.

\begin{figure}[htbp]
  \centering
  \includegraphics[width=5.75cm]{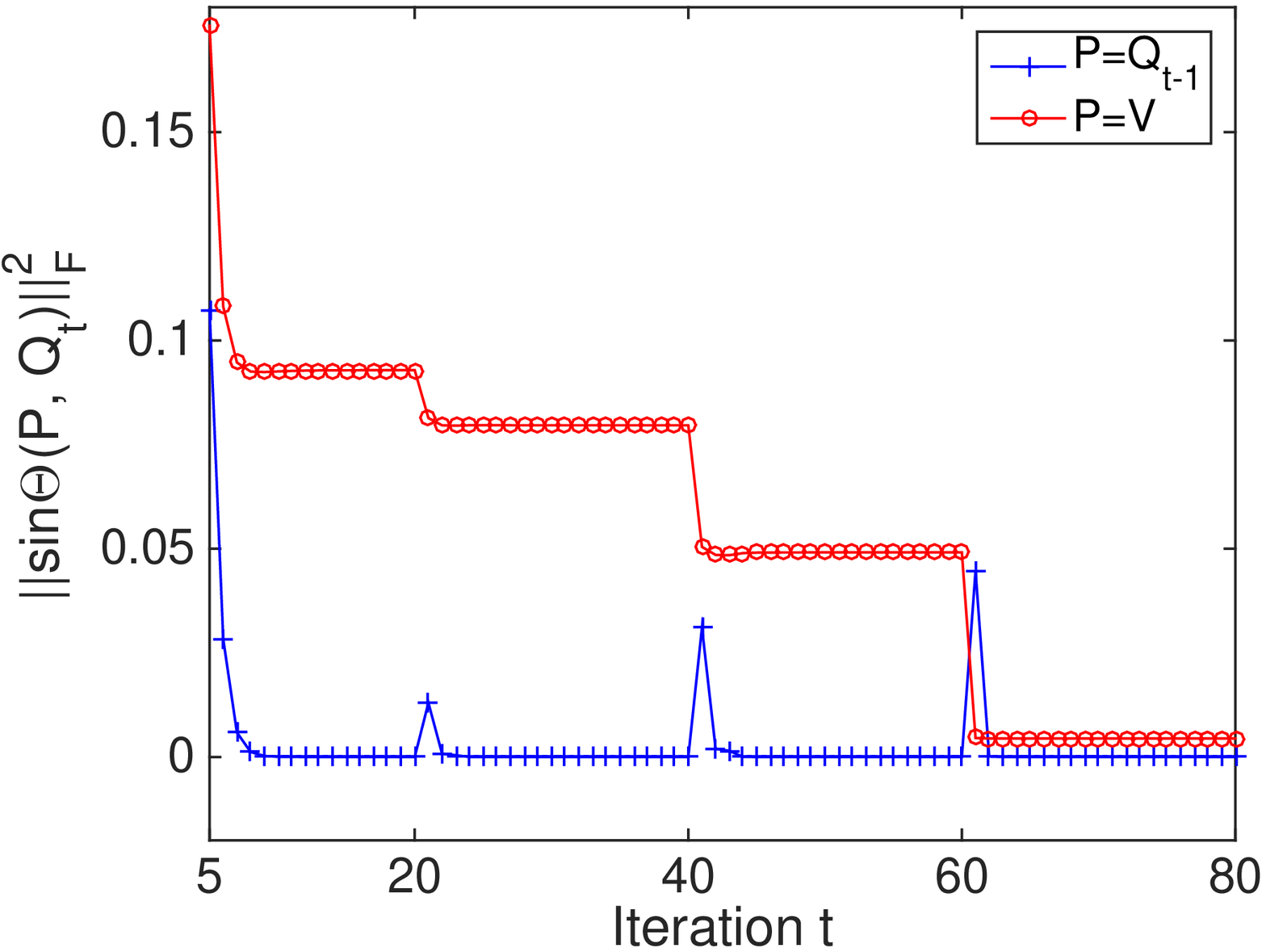}
  \includegraphics[width=9.25cm]{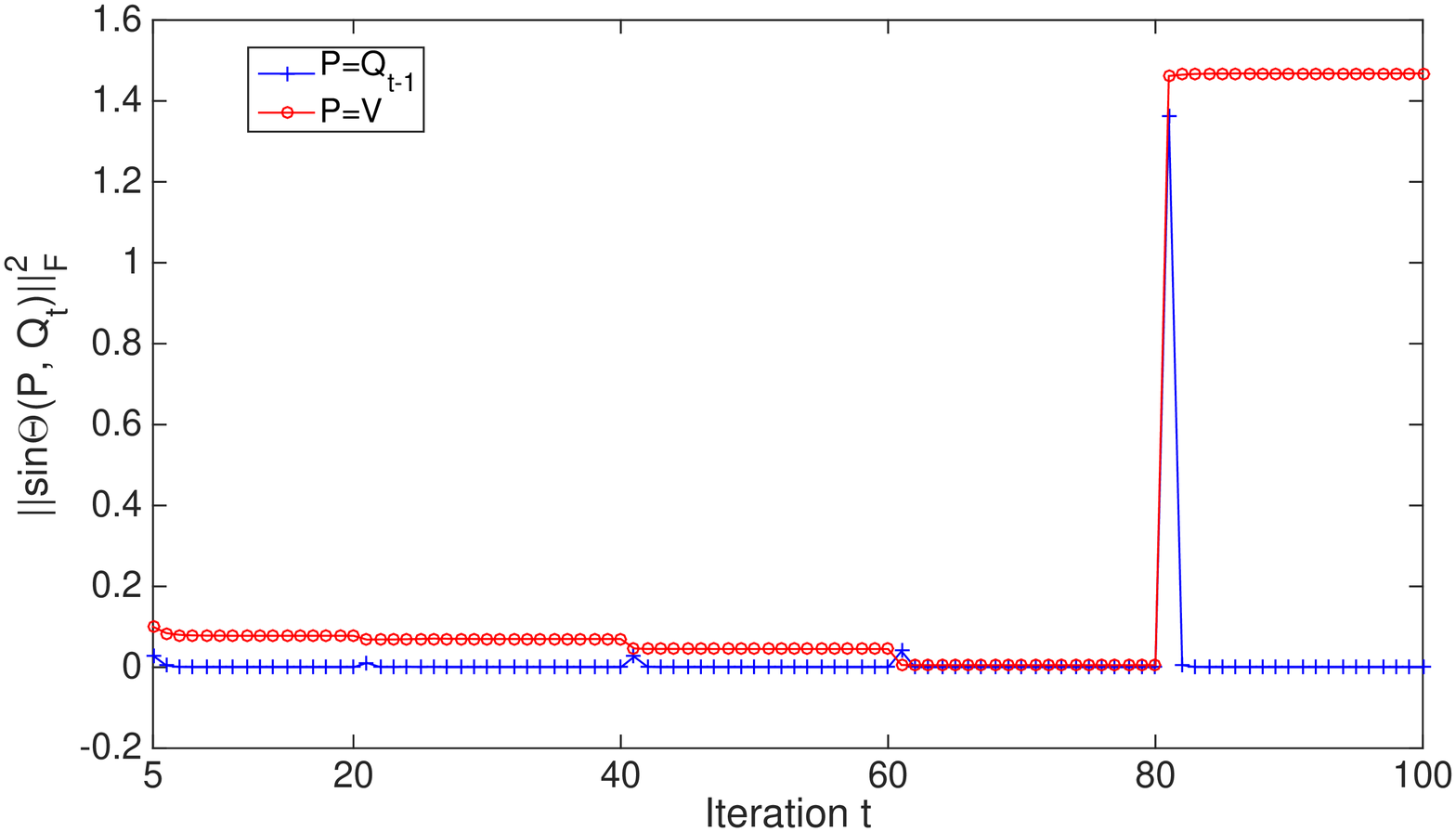}
  \caption{Effect of decreasing the truncation level on the distance between consecutive updates ($Q_t, Q_{t-1}$) and on the distance between the updated matrix and the true eigenvectors ($Q_t, V$). We use the same toy example as described in \cref{sec:sparse}, with the matrix dimension $p=100$ and the true sparsity level $\bar{k}=10$. Left: truncation level starts at $k=100$ and decreases at iteration number 21 ($k=50$), 41 ($k=25$), 61 ($k=10$). Right (over-truncation): For iteration 1-80, same as described in the Right. At iteration 81, truncation level further decreases to $k=5$.}
  \label{fig:init}
\end{figure}

\textbf{Convergence Criterion.} Since we are interested in not only recovering the eigenspace but also recovering the exact eigenvectors, the convergence criterion for the proposed algorithm is based on the size of $\|Q_t-Q_{t-1}\|_2$. Unless specified, the default threshold is set to $10^{-12}$, and the algorithm is forced to terminate after a maximum of 200 iterations.

\subsection{Comparisons on a Simulated Dataset where $\bar{A}$ is known} We first consider a simulated example where we know the true matrix $\bar{A}$ with sparse eigenvectors and we add a small perturbation matrix $E$ to it. Our goal is to recover the true sparse eigenvectors of $\bar{A}$ from the perturbed matrix $A=\bar{A}+E$. Specifically, we use a $1000\times 1000$ positive semidefinite matrix $\bar{A}=V\Sigma V^T$ with the first three columns of $V$ being sparse orthonormal vectors, and the rest columns of $V$ are randomly generated orthonormal vectors such that $V^TV = I$. We consider the following three different cases of the support set: 
\begin{itemize}
    \item Case I: the support sets for the first three eigenvectors are identical. Specifically, for $i=1,2,3$, the first 10 entries of $v_i$ are nonzero and the rest of the entries are zero.
    \item Case II: the support sets partially overlap.
    \item Case III: the support sets are completely non-overlapping. 
\end{itemize}
For all three cases, the eigenvalues are set to be $\lambda_1=1$, $\lambda_2=0.9$, $\lambda_3=0.8$, $\lambda_j=0.1$ for $j=4,\cdots, 1000$. $\rho(E)\approx 0.22$. For each case, we use a fixed $\bar{A}$ and run 1000 trials for each algorithm, and each trial is randomly initialized. We denote the true eigenvectors of $A$ as $v_1, \ v_2, \ v_3$ and the recovered eigenvectors as $u_1, \ u_2, \ u_3$. For each algorithm, we record the inner products of the true and the recovered eigenvectors averaged over \textit{all} trials. We also record the success rate, where a trial is counted as success if all three inner products are greater than 0.99. The recovery rate is measured by the recovery of the full support set, i.e. a trial is counted as a successful recovery if the recovered eigenvector has all nonzeros in the correct indices, otherwise it is counted as failure.


\begin{table}[htbp]
{
  \caption{Results on Simulated Data}\label{tab:simulated}
\begin{center}
  \begin{tabular}{|c|c|c|c|c|c|c|} \hline
   & Algorithms & $|v_1^Tu_1|$ & $|v_2^Tu_2|$ & $|v_3^Tu_3|$ & Success Rate & Recovery Rate
  \\ \hline
  \multirow{4}{6em}{Case I: Completely overlap} 
  & Standard & 0.9912 & 0.9888 & 0.9869 & 0 & 0 \\ \cline{2-7}
  & TPower & 0.9885 & 0.9501 & 0.7881 & 78.7\% & 39.2\% \\ \cline{2-7}
  & TOrth  & 0.9955 & 0.9910 & 0.9872 & \textbf{98.8\%} & 50.0\% \\ \cline{2-7}
  & TOrthT & 0.9935& 0.9910 & 0.9872 & \textbf{98.8\%} & \textbf{51.9\%} \\\hline
  \multirow{4}{6em}{Case II: Partially overlap} 
  & Standard & 0.9916 & 0.9896 & 0.9863 & 0 & 0 \\ \cline{2-7}
  & TPower & 0.9131 & 0.8252 & 0.8141 & 75.8\% & 75.8\% \\ \cline{2-7}
  & TOrth  & 0.9161 & 0.8962 & 0.9600 & \textbf{89.2\%} & 0 \\ \cline{2-7}
  & TOrthT  & 0.8971 & 0.8712 & 0.9540 & 86.8\% & \textbf{86.8\%}\\\hline
  \multirow{4}{6em}{Case III: non-overlap} 
  & Standard & 0.9915 & 0.9892 & 0.9878 & 0 & 0 \\ \cline{2-7}
  & TPower & 0.8690 & 0.7550 & 0.7569 & 66.8\% & 66.8\%\\ \cline{2-7}
  & TOrth  & 0.8470 & 0.8020 & 0.9069 & \textbf{79.4\%} & \textbf{79.4\%}\\ \cline{2-7}
  & TOrthT  & 0.8560 & 0.7960 & 0.9030 & 79.0\% & 79.0\% \\\hline
  \end{tabular}
\end{center}
}
\end{table}

The results show that {TOrth} achieves the best success rate among all three cases, and {TOrthT} achieves a comparable success rate with {TOrth} and has a better sparsity recovery rate than {TPower}. We also observe that in {TPower}, the second and the third recovered vectors $u_2, \ u_3$ have worse quality in terms of distance to the true eigenvector compared to the first recovered vector $u_1$. We explain this phenomenon by doing an error analysis for the deflation scheme. Suppose that $u_j=v_j+z$, where $z$ is the difference between the true and the recovered eigenvector. Without loss of generality, we compute the matrix after one deflation:
\[(I-u_1u_1^T)(\bar{A}+E)(I-u_1u_1^T)\]
\[=(I-v_1v_1^T)\bar{A}(I-v_1v_1^T)+C\bar{A}C+B\bar{A}C+C\bar{A}B+CEC+BEB+BEC+CEB\]
\[:=(I-v_1v_1^T)\bar{A}(I-v_1v_1^T)+E_1\]
where $B=I-v_1v_1^T$ and $C=zz^T-v_jz^T-zv_j^T$. When $u_1$ converges to $v_1$, $z$ is negligible and $C\approx \textbf{0}$. But when $u_1$ diverges, the error accumulates and the perturbation matrix $\rho(E_1)$ to the deflated matrix can be even larger than $\rho(\bar{A})$, thus causing $u_2$ harder to converge.

\subsection{PitProps Dataset}
The PitProps dataset~\cite{jeffers1967two} is one of the classical examples of PCA interpretability and a benchmark to evaluate the performance of sparse PCA algorithms. The dataset contains 180 observations of props of Corsican pine from East Anglia and 13 variables corresponding to the physical properties of the props. The first six principal components obtained by standard PCA accounts for 87\% of total variance. We compute the first six sparse loadings of the data and compare the results with other sparse PCA algorithms. Since the outputs of SPCA algorithms are not guaranteed to be uncorrelated, we measure the performances of sparse PCA algorithms by the proportion of adjusted explained variance (Prop. of AdjVar.) and the cumulative percentage of explained variance (CPEV), as explained below.

\begin{itemize}
    \item \textbf{Prop. of AdjVar.} Suppose $X$ is the data matrix and $V$ are the obtained sparse loadings. The adjusted variance~\cite{zou2006sparse} of the first $m$ principal components $Y=XV$ is computed by: \[\text{AdjVar}(V)=\sum_{j=1}^m R_{jj}^2\]
where $R$ is the upper-triangular matrix obtained by QR factorization of $Y$ and is used to de-correlate between the components. If given the covariance matrix $A=X^TX$, $R$ can be obtained by the Cholesky factorization: $R=\textbf{chol}(V^TAV)$. 

\item \textbf{CPEV.} To account for the non-orthogonality of the loading matrix, CPEV was proposed in~\cite{shen2008sparse} and uses the projection of $X$ onto the $m-$dimensional subspace spanned by the loading vectors $Z$:
\[X_m=XV(V^TV)^{-1}V^T\]
the CPEV is then computed as $\text{Trace}(X_m^T X_m)/\text{Trace}(X^TX)$.
\end{itemize}

We report the results obtained by Algorithm~\ref{alg:OrthT} and compare with the results obtained by TPower~\cite{yuan2013truncated}, GPower~\cite{journee2010generalized}, PathPCA~\cite{d2008optimal}, rSVD~\cite{shen2008sparse} and SPCA~\cite{zou2006sparse} in Table~\ref{tab:pitprops}. 
Overall, TOrthT performs on par with the other algorithms.
\begin{table}[htbp]
{
  \caption{Results on PitProps}  \label{tab:pitprops}
  \begin{center}
  \begin{tabular}{|c|c|c|c|c|} \hline
   Algorithms & Input Parameters & Output Card. & Prop. of AdjVar. & CPEV \\ \hline
  TOrthT  & $K=[7,2,4,3,5,4]$ & 25 & 0.7956 & 0.8487\\ 
  & $K=[6,2,1,2,1,1]$ & 13 & 0.7009 & 0.7528\\ \hline
  TPower &$K=[7,2,4,3,5,4]$ & 25 & 0.7913 & 0.8377\\
  & $K=[6,2,1,2,1,1]$ & 13 & 0.7003 & 0.7585\\ \hline
  GPower$_{\ell_{1}}$ & $\gamma=0.22$; see~\cite{journee2010generalized} & 25 & 0.8083 & 0.8279 \\
  & $\gamma=0.40$ & 13 & 0.7331 & 0.7599\\ \hline
  PathPCA & $K=[7,2,4,3,5,4]$ & 25 & 0.8003 & 0.8438\\ 
  & $K=[6,2,1,2,1,1]$ & 13 & 0.7202 & 0.7700\\ \hline
  rSVD$_{\ell_1}$ & see Shen and Huang~\cite{shen2008sparse} & 25 & 0.8025 & 0.8450\\ \hline
  SPCA & see Zou et al.~\cite{zou2006sparse} & 18 & 0.7575 & 0.8022\\ \hline
  \end{tabular}
\end{center}
}
\end{table}

\subsection{Denoising of Synthetic Signals} In this experiment we follow a similar setting as the denoising experiment from~\cite{jenatton2010structured} and generate signals from the following noisy linear model:
\[u_1 V^1 + u_2 V^2+ u_3 V^3 + \epsilon \in \mathbb{R}^{400}\]
where $V=[V^1, V^2, V^3]\in\mathbb{R}^{400\times 3}$ are sparse and structured dictionary elements organized on a $20\times 20$-dimensional grid and are non-overlapping, as shown in the first row of \cref{fig:synthetic}. Each dictionary element has a structured sparsity consisting of a $10\times 10$ nonzero block. The components of the noise vector $\epsilon$ are independent and identically distributed generated from a normal distribution. The linear coefficients $[u_1, u_2, u_3]$ are generated from a normal distribution:
\[[u_1, u_2, u_3] \sim \mathcal{N}\Bigg(\textbf{0}, \left[ \begin{array}{ccc} 1 & 0 & 0.5 \\ 0 & 1 & 0.5\\ 0.5 & 0.5 & 1 \end{array} \right]\Bigg).\]
We generate $n=250$ signals according to the noisy linear model, and decompose the data matrix to obtain the first three dictionary elements using standard PCA, standard PCA$+$ post truncation (only truncate at the end) and TOrth. The sparsity level is set to be the true sparsity, i.e. $K=[100,100,100]$. The results are shown in \cref{fig:synthetic}. We observe that the standard PCA and simple truncation are not able to recover the original dictionaries, while TOrth finds the structured sparsity in all three elements.
\begin{figure}[h!]
  \includegraphics[width=7.5cm]{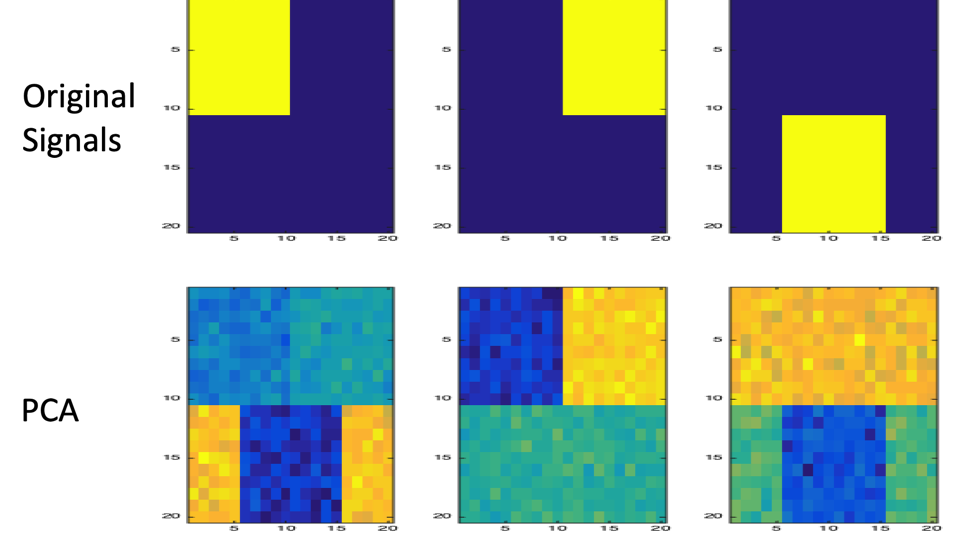}
  \includegraphics[width=7.5cm]{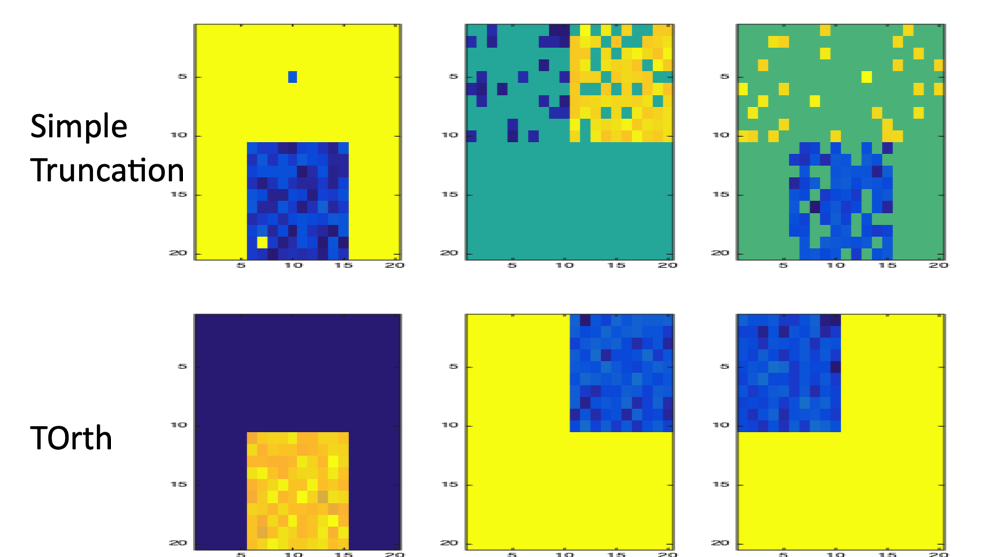}
  \caption{Original signals and signals recovered by PCA, PCA with simple truncation and TOrth.}
  \label{fig:synthetic}
\end{figure}

\subsection{Sea Surface Temperature Example}
The Sea Surface Temperature dataset (SST) records the weekly means of satellite ocean temperature data over $360 \times 180 = 64,800$ grid points from 1990 to present~\cite{reynolds2002improved}. The El Ni\~no Southern Oscillation (ENSO) is defined as any sustained temperature anomaly above running mean temperature with a duration of 9 to 24 months. The canonical El Niño is associated with a narrow band of warm water off coastal Peru, as recovered by TOrth shown in \cref{fig:SST_PC} left. Standard PCA, as shown in \cref{fig:SST_PC} right, is unable to separate this band from a global weather pattern across the Pacific and Atlantic.   

We use this example to  demonstrate the computational efficiency of our algorithm. The SST data dimension is $1455 \times 64800$, where $n=1455$ is the number of temporal snapshots and $p=64800$ is the spatial grid points in each snapshot. We compute the fourth mode that is associated with the canonical El Niño, so we set $m=4$ and compare with other block (sparse) PCA algorithms. The time complexity and actual running time are recorded in \cref{tab:SST}. Truncated SVD serves as a baseline as it computes all singular vectors and singular values and is more expensive than algorithms that only computes the leading vectors. The other approaches, although comparable to TOrth in time complexity, require longer running time since the Polar decomposition via SVD~\cite{journee2010generalized} is more expensive than the QR decomposition. The methods we compared to (GPower~\cite{journee2010generalized}, Variable Projection SPCA~\cite{erichson2020sparse}) are already among the top performers in terms of computational speed. For example, the elastic net SPCA algorithm~\cite{zou2006sparse} requires $O(np^2+p^3)$ and takes at least $\times 10$ running time compared to Variable Projection SPCA~\cite{erichson2020sparse} on the SST dataset.
\begin{figure}[h!]
  \includegraphics[width=7.4cm]{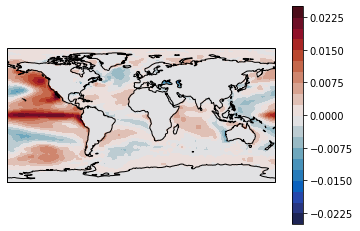}\hspace{0.25cm}
  \includegraphics[width=7.4cm]{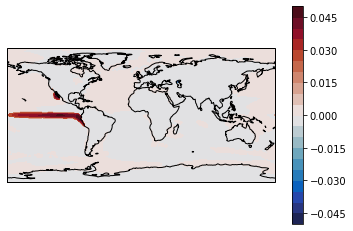}
  \caption{The fourth mode recovered by standard PCA (Left) and by TOrth (Right).}
  \label{fig:SST_PC}
\end{figure}

\begin{table}[htbp]
{
  \caption{Comparisons on Complexity and Running Time}  \label{tab:SST}
  \begin{center}
  \small{
  \begin{tabular}{|c|c|c|c|c|} \hline
    & SVD  & VarProj & GPower$_m$ & TOrth \\ \hline
    $n>p$ & $O(np^2)$  & $O(np^2)+kO(mpn+m^2p)$ & $kO(mpn+m^2p)$ & $kO(mpn+m^2p)$  \\ \hline
    $n<p$ & $O(pn^2)$  & $O(pn^2)+kO(mpn+m^2p)$ & $kO(mpn+m^2p)$ & $kO(mpn+m^2p)$ \\ \hline
   Running Time & 31.30 & 44.79 & 17.19 & 5.96 \\ \hline
  \end{tabular}
  }
\end{center}
}
\end{table}

\subsection{Classification Example on the MNIST dataset}
We apply \cref{alg:TOrth} to the MNIST handwritten digit dataset and compare the classification performance with that of standard PCA. The MNIST dataset has 70,000 samples and each samples is a 2D image with $28\times 28$ pixels,  yielding 748 features. We first use the default train-test split where 60,000 samples are used as the training set and 10,000 samples as the test set. Applying a k-nearest neighbor classifier (KNN) on the original data gives a prediction error of $3\%$, where the number of nearest neighbors is chosen to be 3 by cross-validation.

We apply standard PCA and TOrth to reduce the dimension from 748 to various subspace dimension $p$. With PCA, applying KNN on the projected test data achieves the lowest prediction error of $2.47\%$ when $p=60$, and the performance of KNN starts to decline due to the curse of dimensionality. Fixing $p=60$, results of TOrth with different $k$ are shown in \cref{tab:mnist}. We observe that by using only $k=20$ in each loading vectors and 60 loading vectors, TOrth is able to achieve a prediction error comparable to KNN applied to raw data with full dimension. \cref{fig:mnistfull} displays the top 30 loading vectors obtained by PCA and TOrth, where loading vectors obtained by TOrth captures more local features and includes fractions/strokes of digits.

\begin{table}[htbp]
{
  \caption{Prediction error on MNIST}  \label{tab:mnist}
\begin{center}
  \begin{tabular}{|c|c|c|c|c|c|c|c|c|} \hline
   k & 10 & 20 & 40 & 80 & 160 & 320 & 640 & 748 $(p)$\\ \hline
   Prediction error (\%)  & 4.26 & 2.95 & 2.78 & 2.64 & 2.63& 2.49 & 2.48 & 2.47\\ \hline
  \end{tabular}
\end{center}
}
\end{table}

\begin{figure}[h!]
  \centering
  \includegraphics[width=7cm]{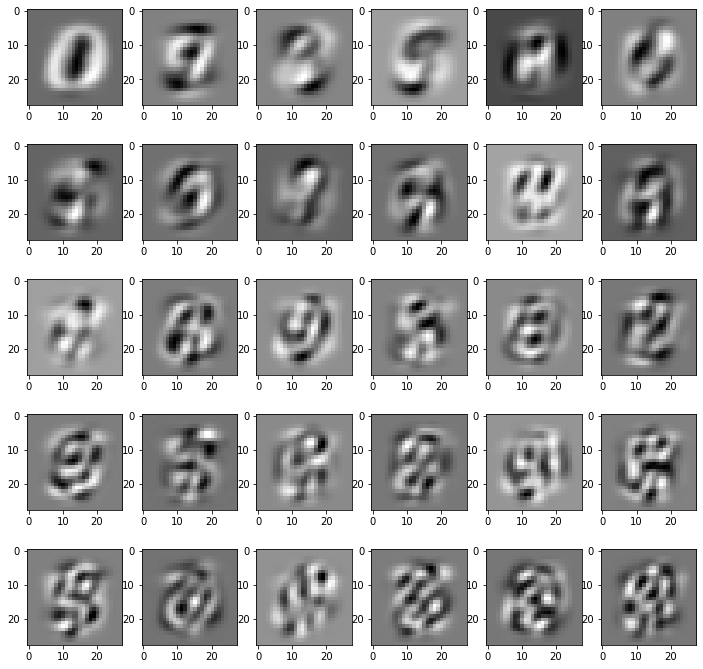}\hspace{0.3in}
  \includegraphics[width=7cm]{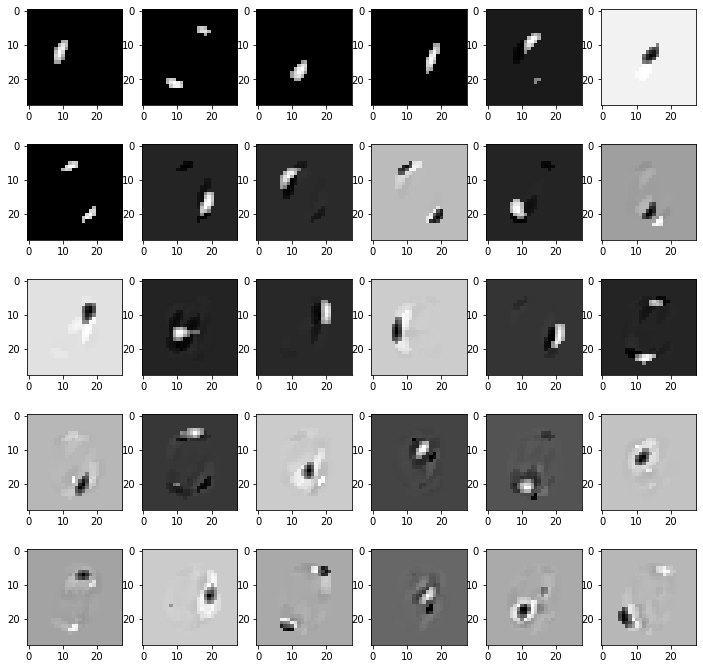}
  \caption{Top 30 loading vectors of MNIST training set. Left: loading vectors obtained by standard PCA. Right: loading vectors obtained by \cref{alg:TOrth}, with $k_i=20, \forall i$.}
  \label{fig:mnistfull}
\end{figure}

\subsection{20 newsgroup dataset} We take a sample of the 20 newsgroup dataset~\cite{lang1995newsweeder}, with binary occurrence data for 100 keywords across 16242 postings. The postings come from 4 general topics: computer, recreation, science and talk. We apply sparse PCA using {TOrthT} with sparsity level $k=10$ and the standard PCA on the centered data. \cref{tab:20ngcoeff} shows the 10 nonzero entries in the coefficient vector for the first two PCs. The keywords associated with the first PC are more relevant to religion and politics, and the keywords associated with the second PC are more relevant to computers. We also project the data onto the 2D subspace spanned by PC1 and PC2, as shown in~\cref{fig:20ng}. We observe that with the loading vectors obtained from {TOrthT}, the projections of the ``computer"-themed data are dense in PC2 and sparse in PC1, while the ``talk"-themed data projections are dense in PC1 and sparse in PC2. In contrast, projections onto the normal PCs are clustered and are dense in both directions, and lacks physical interpretations.

\begin{table}[htbp]
{
\caption{Nonzero entries in the coefficient vector obtained by TOrthT} \label{tab:20ngcoeff}
\begin{center}
\begin{tabular}{ |c|c c c c c| } 
 \hline
 PC1 & Question & Fact & Problem & Course & Case \\
 & World & God & Number & Human & Government \\ 
 \hline
 PC2 & Help & Email & Problem & System & Windows \\
 & Program & University & Computer & Software & Files \\ 
 \hline
\end{tabular}
\end{center}
}
\end{table}

\begin{figure}[h]
     \centering
     \begin{subfigure}[b]{0.4\textwidth}
         \centering
         \includegraphics[width=\textwidth]{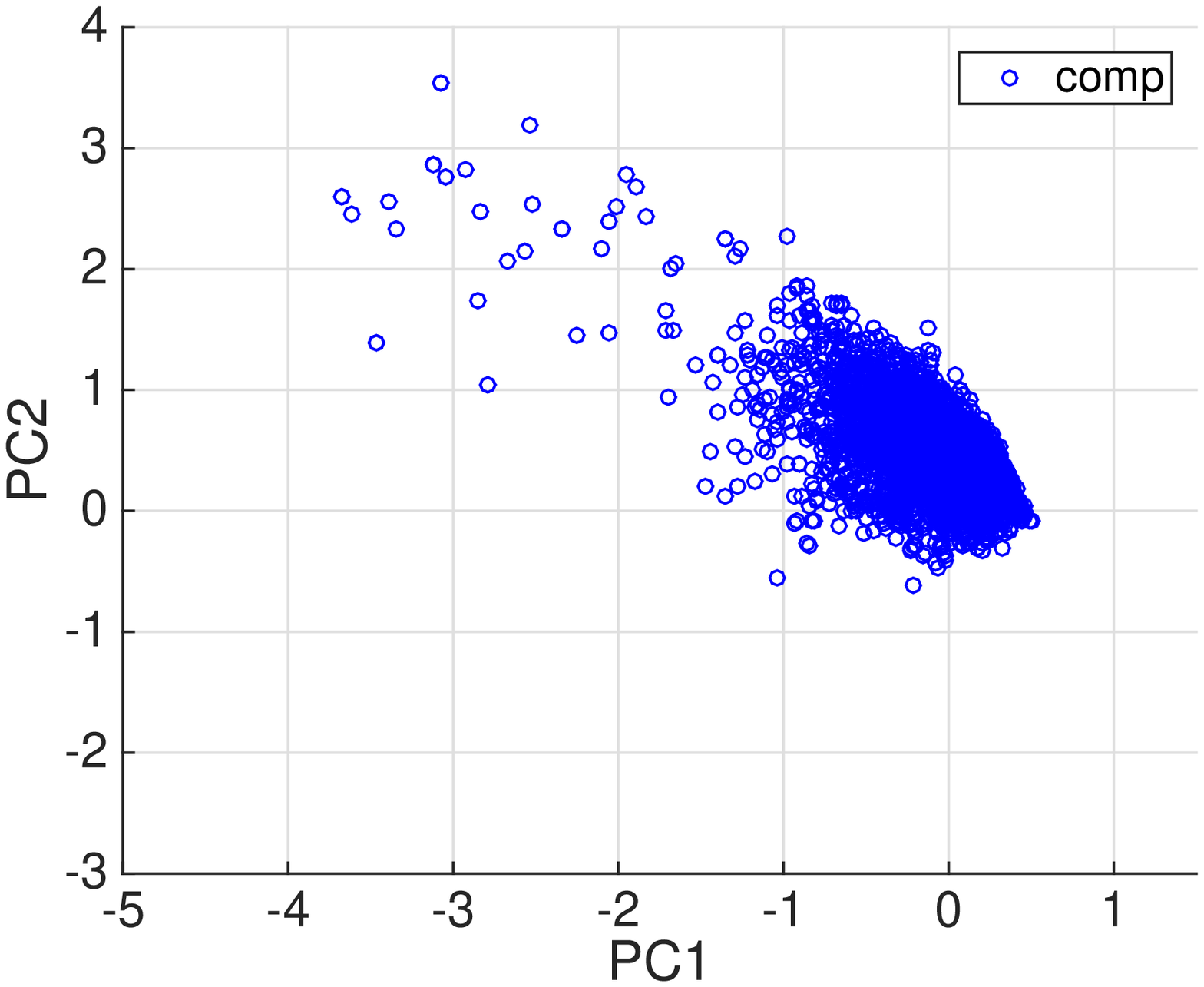}
         \caption{}
         \label{fig:comp_dense}
     \end{subfigure}
     \begin{subfigure}[b]{0.4\textwidth}
         \centering
         \includegraphics[width=\textwidth]{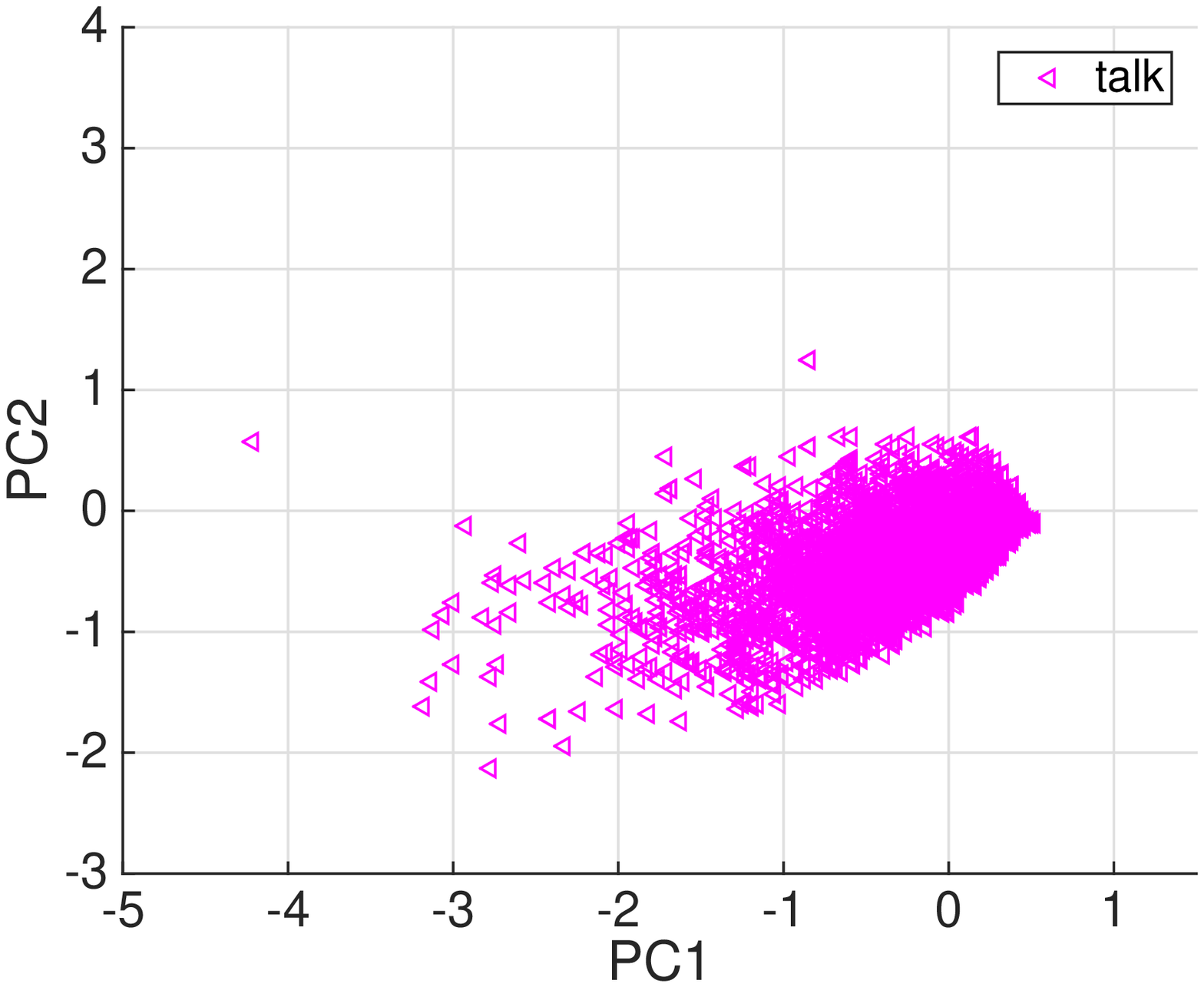}
         \caption{}
         \label{fig:talk_dense}
     \end{subfigure}
     \begin{subfigure}[b]{0.4\textwidth}
         \centering
         \includegraphics[width=\textwidth]{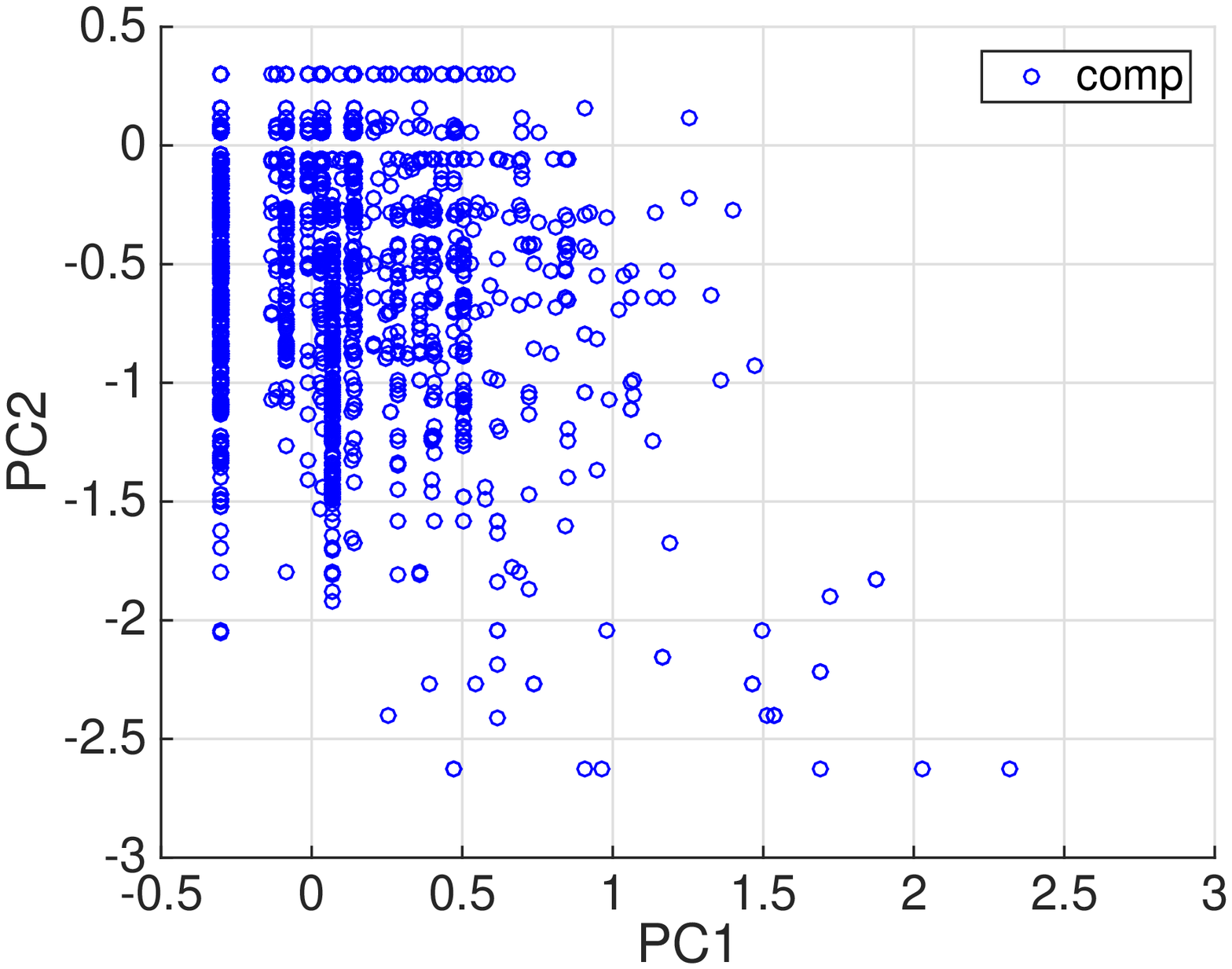}
         \caption{}
         \label{fig:comp_sparse}
     \end{subfigure}
     \begin{subfigure}[b]{0.4\textwidth}
         \centering
         \includegraphics[width=\textwidth]{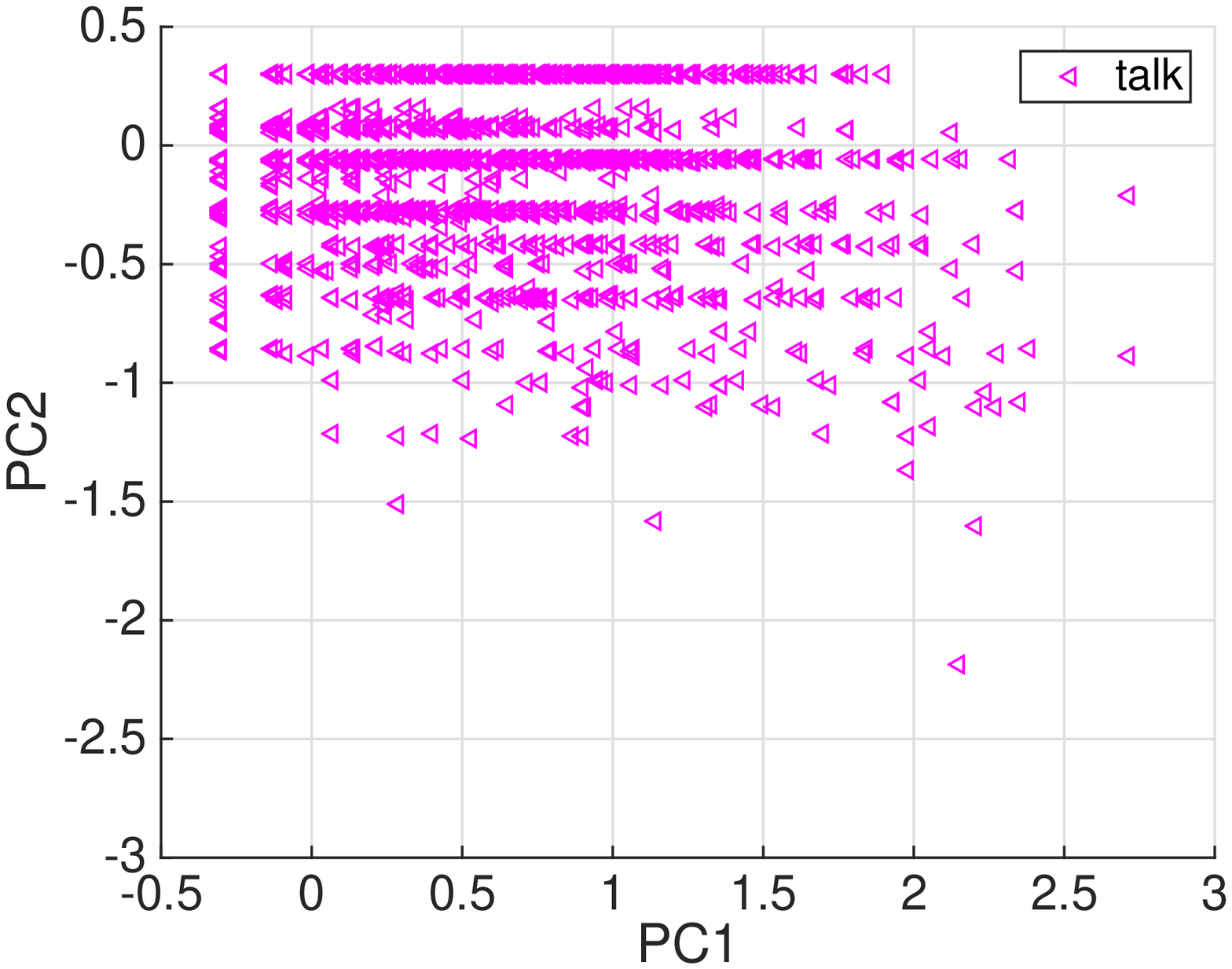}
         \caption{}
         \label{fig:talk_sparse}
     \end{subfigure}
        \caption{Projection of the ``computer" and ``talk"-themed data onto PC1 and PC2. (a) and (b): projections onto PC directions obtained by PCA. (c) and (d): projections onto PC directions obtained by {TOrthT}.}
        \label{fig:20ng}
\end{figure}

\section{Conclusions}
\label{sec:conclusions}

In this paper, we have proposed two algorithms based on the orthogonal iteration to recover sparse eigenvectors, and established convergence analyses. Our algorithms can be easily implemented and work efficiently on a wide range of data science applications, while achieving comparable or superior performance compared with existing algorithms. Compared to its single-vector counterpart, the block scheme is more robust to random initialization and achieves better accuracy and sparse recovery rate. 

There are still many open problems in this area of research. One challenge for the block approach is how to maximally preserve orthogonality while achieving sparsity. Another problem is choosing a proper initialization strategy that leads to better and faster convergence. Lastly, different truncation schemes, including probabilistic rather than deterministic approaches, can be explored to preserve the true support set and avoid truncation error in the future.

\section*{Acknowledgments}
\newpage

\section*{Appendix}
We give the proof for~\cref{eq:Fnorm} and~\cref{eq:onestep}, which is based on [Theorem 8.1.10] of~\cite{van1983matrix}.
\begin{theorem}
Let $P$ be the matrix of eigenvectors corresponding to the $m-$largest eigenvalues of $\bar{A}$. Assume $\lambda_m>\lambda_{m+1}$. Define $\gamma :=\frac{\lambda_{m+1}}{\lambda_{m}}$. Then the matrices $Q_t$ generated by the standard orthogonal iteration satisfy:
\[\|\sin\Theta (P, Q_t)\|_F \leq \gamma \frac{\|\sin\Theta (P, Q_{t-1})\|_F}{\sqrt{1-\|\sin\Theta (P, Q_{t-1})\|_2^2}},\]
assuming $\|\sin\Theta (P, Q_{t-1})\|_2<1$.
\end{theorem}

\begin{proof}
In the standard orthogonal iteration,
\begin{equation}\label{eq:QR}
    \bar{A}Q_{t-1}=Q_t R_t.
\end{equation}
We can decompose $Q_t$ as $Q_t=PX_t+P^{\perp}Y_t$, where $P^{\perp}$ is the orthogonal complement of $P$ and its columns are the eigenvectors of $\bar{A}$ corresponding to the $m+1, \cdots p$ eigenvalues, i.e. $\bar{A}P=P\Lambda_m$, $\bar{A}P^{\perp}=P^{\perp}\Lambda'$, where $\Lambda'=\text{diag}(\lambda_{m+1}, \cdots, \lambda_p)$. We have the following equations:
\begin{equation}\label{eq:PQ}
    X_t = P^TQ_t, \quad Y_t = {P^\perp}^TQ_t \Leftrightarrow \bcm X_t\\ Y_t\ecm =\bcm P & P^{\perp} \ecm^T Q_t .
\end{equation}
By applying the thin CS decomposition \cite{van1983matrix} on~\cref{eq:PQ} we have
\begin{equation}\label{eq:minmax}
    1 = \sigma_{\min}(X_t)^2+\sigma_{\max}(Y_t)^2 ,
\end{equation}
\begin{align}
    \bcm P & P^{\perp} \ecm \bcm \Lambda_m & \ \\ \ &\Lambda'\ecm \bcm P^T\\ {P^{\perp}}^T\ecm Q_{t-1} &= Q_t R_t \text{ by~\cref{eq:QR}} ,\\
    \bcm \Lambda_m & \ \\ \ &\Lambda'\ecm \bcm P^T\\ {P^{\perp}}^T\ecm Q_{t-1} &= \bcm P^T\\ {P^{\perp}}^T\ecm Q_t R_t ,\notag\\ 
    \bcm \Lambda_m & \ \\ \ &\Lambda'\ecm \bcm X_{t-1}\\ Y_{t-1}\ecm &= \bcm X_t\\Y_t\ecm R_t \notag \\
    \Rightarrow \Lambda_m X_{t-1} = X_t R_t, &\quad \Lambda' Y_{t-1} = Y_t R_t. \label{eq:Yt}
\end{align}
Assuming $R_t$ and $X_{t-1}$ are nonsingular, we obtain the following equations from~\cref{eq:Yt}: 
\begin{align}
    Y_t &= \Lambda' Y_{t-1} R_t^{-1}, \label{eq: Yt}\\
    R_t^{-1} & = (\Lambda_m X_{t-1})^{-1} X_t = X_{t-1}^{-1} \Lambda_m^{-1}  X_t, \\
\|R_t^{-1}\|_2 &\leq \|X_{t-1}^{-1}\|_2 \|\Lambda_m^{-1}\|_2 \|X_t\|_2 \leq \frac{1}{\sqrt{1-\|Y_{t-1}\|_2^2}} \frac{1}{\lambda_m} \label{eq:RtNorm}
\end{align}
The last inequality in~\cref{eq:RtNorm} comes from~\cref{eq:minmax} and the fact that
\[\|X_{t-1}^{-1}\|_2\leq \frac{1}{\sigma_{\min}(X_{t-1})} = \frac{1}{\sqrt{1-\sigma_{\max}(Y_{t-1})^2}} .\]
Based on~\cref{eq: Yt} and \cref{eq:RtNorm}, we arrive at the bound given in \cref{eq:onestep}:
\begin{align}
    \|\sin\Theta (P,Q_t)\|_F = \|Y_t\|_F &\leq \|\Lambda'\|_2\|Y_{t-1}\|_F\|R_t^{-1}\|_2\\
    &\leq \lambda_{m+1} \cdot \|Y_{t-1}\|_F \cdot \frac{1}{\sqrt{1-\|Y_{t-1}\|_2^2}} \cdot \frac{1}{\lambda_m} \\
    &= \gamma \frac{\|\sin\Theta (P,Q_{t-1})\|_F}{\sqrt{1-\|\sin\Theta (P,Q_{t-1})\|_2^2}} .
\end{align}

Similarly, based on
\begin{equation}
    \Lambda_m^t X_{0} = X_t (R_t \cdots R1), \quad {\Lambda'}^t Y_{0} = Y_t (R_t \cdots R1) ,
\end{equation}
we can derive the bound in~\cref{eq:Fnorm}:
\begin{align}
    \|\sin\Theta (P,Q_t)\|_F = \|Y_t\|_F &\leq \|\Lambda'\|^t_2\|Y_0\|_F\|(R_t\cdots R1)^{-1}\|_2\\
    &\leq \lambda_{m+1}^t \cdot \|Y_{0}\|_F \cdot \frac{1}{\sqrt{1-\|Y_{0}\|_2^2}} \cdot \frac{1}{\lambda_m^t} \\
    &= \gamma^t \frac{\|\sin\Theta (P,Q_{0})\|_F}{\sqrt{1-\|\sin\Theta (P,Q_{0})\|_2^2}} .
\end{align}

\end{proof}
\newpage

\newpage
\bibliographystyle{siamplain}
\bibliography{list}
\end{document}